\documentclass{article}
\usepackage[margin=1.2in]{geometry}

\usepackage{hyperref}
\usepackage{tablefootnote}
\usepackage{graphicx}
\usepackage{epstopdf}
\usepackage{subcaption}
\usepackage{amsthm}
\usepackage{dsfont}
\usepackage{amssymb}
\usepackage{enumerate}
\usepackage{graphicx}
\usepackage{float}
\usepackage{bbm}
\usepackage{amsmath}
\usepackage{comment}
\usepackage{hyperref}
\usepackage{listings}
\usepackage{color}
\usepackage[normalem]{ulem}
\usepackage{multirow}
\usepackage{multicol}
\usepackage[dvipsnames]{xcolor}
\usepackage[lined,boxed,commentsnumbered]{algorithm2e}
\usepackage{stmaryrd}

\allowdisplaybreaks

\newtheorem{theorem}{Theorem}[section]

\newtheorem{Lemma}[theorem]{Lemma}

\newtheorem{definition}[theorem]{Definition}

\newtheorem{claim}[theorem]{Claim}

\theoremstyle{definition}
\newtheorem{Example}[theorem]{Example}

\newtheorem{Remark}[theorem]{Remark}

\newcommand{\tens}[1]{\mathcal{#1}}

\newcommand{\rank}{{\rm rank\,}}

\newcommand{\argmin}{\text{argmin}}

\newcommand{\kron}{\otimes}  
\newcommand{\khatri}{\odot} 
\newcommand{\hadam}{\boxdot} 
\newcommand{\out}{\pmb\kron}  
\usepackage[dvipsnames]{xcolor}

\begin{document}
\title
{HOSVD-Based  Algorithm for Weighted Tensor Completion
}

 \author{Zehan Chao, Longxiu Huang and Deanna Needell}
 \date{}



\maketitle

\begin{abstract}
Matrix completion, the problem of completing missing entries in a data matrix with low-dimensional structure (such as rank), has seen many fruitful approaches and analyses. Tensor completion is the tensor analog that attempts to impute missing tensor entries from similar low-rank type assumptions. In this paper, we  study the tensor completion problem when the sampling pattern is deterministic and possibly non-uniform. We first propose an efficient weighted Higher Order Singular Value Decomposition (HOSVD) algorithm for the recovery of the underlying low-rank tensor from noisy observations and then derive the error bounds under a properly weighted metric. Additionally,  the efficiency and accuracy of our algorithm  are both tested using synthetic and real datasets in numerical simulations.
\end{abstract}


\section{Introduction}

In many data-rich domains such as computer vision, neuroscience, and social networks, 
tensors have emerged as a powerful paradigm for handling the data deluge. In recent years, tensor analysis has gained more and more attention. 
To a certain degree, tensors can be viewed as the generalization of matrices to higher dimensions, and thus multiple questions from matrix analysis extend naturally to tensors.  
Similar to matrix decomposition, the problem of tensor decomposition (decomposing an input tensor into several less complex components) has been widely studied both in theory and application (see e.g.,  \cite{HF1927,kolda2009tensor,ZOIA2018}).  Thus far, the problem of low-rank tensor completion, which aims to complete missing or unobserved entries of a low-rank tensor, is one of the most actively studied problems (see e.g.,  \cite{GCZS16,LMWY12,LSCT14,SGCH19}). It is noteworthy that, as caused by various unpredictable or unavoidable reasons, multidimensional datasets are commonly raw and incomplete, and thus often only a small subset of entries of tensors are available. 
It is, therefore, natural to address the above issue using tensor completion in modern data-driven applications, in which data are naturally represented as a tensor, such as image/video inpainting \cite{kressner2014low,LMWY12}, link-prediction \cite{ermics2015link}, and recommendation systems \cite{symeonidis2008tag}, to name a few.

In the past few decades, the matrix completion problem, which is a special case of tensor completion, has been extensively studied. In matrix completion, there are mature algorithms \cite{CCS10}, theoretical foundations \cite{ CZ2016, CP2010,CR9} and various applications  \cite{AFSU7,CCWY2019,GL2011,LV2009} that pave the way for solving the tensor completion problem in high-order tensors. 
Recently, \mbox{Foucart et al. \cite{foucart2019weighted}} proposed a simple algorithm for matrix completion for general deterministic sampling patterns, and raised the following questions:
given a deterministic sampling pattern $\Omega$ and corresponding (possibly noisy) observations of the matrix entries, what type of recovery error can we expect? In what metric? How can we efficiently implement recovery?
These were investigated in \cite{foucart2019weighted} by introducing an appropriate \textit{weighted} error metric for matrix recovery of the form $\|H\hadam(\widehat{M}-M)\|_F$, where $M$ is the true underlying low-rank matrix, $\widehat{M}$ refers to the recovered matrix, and $H$ is a best rank-1 matrix approximation for the sampling pattern $\Omega$.
In this regard, similar questions arise for the problem of tensor completion with deterministic sampling patterns. Unfortunately, as is often the case, moving from the matrix setting to the tensor setting presents non-trivial challenges, and notions such as \textit{rank} and SVD need to be re-defined and re-evaluated. We address these extensions for the completion problem here.

Motivated by the matrix case, we propose an appropriate \textit{weighted} error metric for tensor recovery of the form $\|\mathcal{H}\hadam(\widehat{\mathcal{T}}-\mathcal{T})\|_F$, where $\mathcal{T}$ is the true underlying low-rank tensor, $\widehat{\mathcal{T}}$ is the recovered tensor, and $\mathcal{H}$ is an appropriate weight tensor. For the existing work, the error is only limited to the form $\|\widehat{\mathcal{T}}-\mathcal{T}\|_F$, which corresponds to the case that all the entries of $\mathcal{H}$ are 1, where $\mathcal{H}$ can be considered to be a {CP} 
 rank-1 tensor. It motivates us to rephrase the questions mentioned above as follows.

{\bf Main questions. } 
Given a sampling pattern $\Omega$, and noisy observations $\mathcal{T}+\mathcal{Z}$ on $\Omega$, for what rank-one weight tensor $\mathcal{H}$ can we efficiently find a tensor $\widehat{\mathcal{T}}$ so that \mbox{$\|\mathcal{H}\hadam(\widehat{\mathcal{T}}-\mathcal{T})\|_F$} is small compared to $\left\|\mathcal{H}\right\|_F$? And how can we efficiently find such weight tensor $\mathcal{H}$, or determine that a fixed $\mathcal{H}$ has this property?
\subsection{Contributions}
Our main goal is to provide an algorithmic tool, theoretical analysis, and numerical results that address the above questions.  In this paper, we  propose a simple weighted Higher Order Singular Value Decomposition (HOSVD) method. Before we implement the weighted HOSVD algorithm, we first   appropriately approximate the sampling pattern $\Omega$ with a rank one tensor $\mathcal{H}$.  We can achieve high accuracy if $\|\mathcal{H}-\mathcal{H}^{(-1)}\hadam\boldsymbol{1}_{\Omega}\|_F$ is small, where $\mathcal{H}^{(-1)}$ denotes the element-wise inverse. Finally, we present empirical results on synthetic and real datasets. The simulation results show that when the sampling pattern is non-uniform, the use of weights in the weighted HOSVD algorithm is essential, and the results of the weighted HOSVD algorithm can provide a very good initialization for the total variation minimization algorithm which can dramatically reduce the iterative steps without lose the accuracy.  {In doing so, we extend the weighted matrix completion results of \cite{foucart2019weighted} to the tensor setting.}

\subsection{Organization}
The paper is organized as follows. In Section \ref{section:tcp}, we give a brief review of  related work and concepts for tensor analysis, instantiate notations, and state the tensor completion problem under study. Our main results are stated in Section \ref{sec:results} and the proofs are provided in Appendices \ref{section:proofFgub} and  \ref{section:proofFHOSVD}.  The numerical results are provided and discussed in Section \ref{section:simulations}. 

\section{Related Work, Background, and Problem Statement}\label{section:tcp}
In this section, we give a brief overview of the works that are related to ours, introduce some necessary background information about tensors, and finally give a formal statement of tensor completion problem under study.  The related work can be  divided into two lines: that based on matrix completion problems, which leads to a discussion of weighted matrix completion and related work, and that based on tensor analysis, in which we focus on CP and Tucker decompositions. 

\subsection{Matrix Completion}
The matrix completion problem is to determine a complete $d_1\times d_2$ matrix $M$ from its partial entries on a subset $\Omega\subseteq[d_1]\times [d_2]$.  We use {$\boldsymbol{1}$}$_{\Omega}$ 
to denote the matrix whose entries are $1$ on $\Omega$ and $0$ elsewhere so that the entries of $M_{\Omega}=\boldsymbol{1}_{\Omega}\hadam M$ are equal to those of the matrix $M$ on $\Omega$, and are equal to $0$ elsewhere, where $\hadam$ denotes the Hadamard product.
There are various works that aim to understand matrix completion  with respect to the sampling pattern $\Omega$. For example, the works in \cite{BJ2014,HSS2014,LLR2016} relate the sampling pattern $\Omega$ to a graph whose adjacency matrix is given by $\boldsymbol{1}_{\Omega}$ and show that as long as the sampling pattern $\Omega$ is suitably close to an expander, efficient recovery is possible when the given  matrix $M$ is sufficiently incoherent.  Mathematically, the task of understanding when there exists a unique  low-rank matrix $M$ that can complete $M_{\Omega}$ as a function of the sampling pattern $\Omega$ is very important. In \cite{PBN2016}, the authors give conditions on $\Omega$ under which there are only finitely many low-rank matrices that agree with $M_{\Omega}$, and the work of \cite{SXZ2018} gives a condition under which the matrix can be locally uniquely completed. The work in \cite{AAW2017} generalized the results of \cite{PBN2016,SXZ2018} to the setting where there is sparse noise added to the matrix. The works \cite{AWA2017,PN2016} study when  rank estimation is possible as a function of a deterministic pattern $\Omega$. Recently, \cite{C2019} gave a combinatorial condition on $\Omega$ that characterizes when a low-rank matrix can be recovered up to a small error in the Frobenius norm from observations in $\Omega$ and showed that nuclear  minimization will approximately recover $M$ whenever it is possible, where the \textit{nuclear norm }
 of $M$ is defined as $\|M\|_*:=\sum_{i=1}^{r}\sigma_i $
with $\sigma_1,\cdots,\sigma_r$  the non-zero singular values of $M$.

All the works mentioned above are in the setting where recovery of the entire matrix is possible, but in many cases full recovery is impossible. Ref. 
\cite{KTT2012}  uses an algebraic approach to answer the question of when an individual entry can be completed. There are many works (see e.g., \cite{EYW2018,NW2012}) that introduce a weight matrix for capturing the recovery results of the desired entries.
The work \cite{HSS2014} shows that, for any weight matrix, $H$, there is a deterministic sampling pattern $\Omega$ and an algorithm that returns $\widehat{M}$ using the observation $M_{\Omega}$ such that $\|H\hadam(\widehat{M}-M)\|_F$ is small. The work \cite{LS2013}  generalizes the algorithm in \cite{HSS2014} to find the ``simplest'' matrix that is correct on the observed entries. Succinctly, their works give a way of measuring which deterministic sampling patterns, $\Omega$, are ``good'' with respect to a weight matrix $H$. In contrast to these two works, \cite{foucart2019weighted} is interested in  the problem of whether one can find a weight matrix $H$ and  create an efficient algorithm to find an estimate $\widehat{M}$ for an underlying low-rank matrix $M$ from a    sampling pattern $\Omega$ and   noisy samples $M_{\Omega}+Z_{\Omega}$ such that $\|H\hadam(\widehat{M}-M)\|_F$ is  small.

 {In particular, one of our theoretical results is that we generalize the upper bounds for weighted recovery of low-rank matrices from deterministic sampling patterns in \cite{foucart2019weighted} to the upper bound of tensor weighted recovery. The details of the connection between our result and the matrix setting result in \cite{foucart2019weighted} is discussed in Section \ref{sec:results}. }

\subsection{Tensor Completion Problem}
Tensor completion is  the problem of filling in the missing elements of partially observed tensors. 
Similar to the matrix completion problem, \textit{low rankness} is often a necessary hypothesis to restrict the degrees of freedom of the missing entries for the tensor completion problem. Since there are multiple definitions of the rank of a tensor, this completion problem has several variations. 

The most common tensor completion problems \cite{GRY2011,LMWY12} may be summarized as follows (we will define the different ranks subsequently, see further on in this section).
\begin {definition}[Low-rank tensor completion (LRTC), \cite{SGCH19}] 
Given a low-rank (CP rank, Tucker rank, or other ranks) tensor $\mathcal{T}$  and sampling pattern $\Omega$, 
the low-rank completion of $\mathcal{T}$ is given by the solution of the following optimization problem:
\begin{eqnarray}\label{eqn:g-lrtc}
&&\underset{\mathcal{X}}{\text{min} } \text{ rank}_*(\mathcal{X})\nonumber\\
&&\text{subject to }\mathcal{X}_{\Omega}=\mathcal{T}_{\Omega},
\end{eqnarray}
where $\text{rank}_*$ denotes the specific tensor rank assumed at the beginning.
\end {definition}
In the literature, there are many variants of  LRTC but most of them are based on the following  questions:
\begin{enumerate}
    \item [(1)]
    What type of the rank should one use (see e.g., \cite{AW2017fundamental,BM2016,JO2014})? 
    \item[(2)]
    Are there any other restrictions based on the observations that one can assume (see e.g., \cite{GQ2014,LMWY12,MHWG2014})?
    \item[(3)] Under what conditions can one expect to achieve a unique and exact completion (see e.g., \cite{AW2017fundamental})?
\end{enumerate}

In the rest of this section,   we instantiate some notations and review  basic operations and  definitions  related to tensors. Then some tensor decomposition-based algorithms for tensor completion are stated. Finally, a formal problem statement under study will \mbox{be presented}.

\subsubsection{Preliminaries and Notations} \label{sec: notation}
 Tensors, matrices, vectors, and scalars are denoted in different typeface for clarity below. In the sequel, calligraphic boldface capital letters are used for tensors,  capital letters are used for matrices, lower boldface letters for vectors, and regular letters for scalars.  The set of the first $d$ natural numbers is denoted by $[d]:=\{1,\cdots,d\}$. Let $\mathcal{X}\in\mathbb{R}^{d_1\times\cdots\times d_n}$ and $\alpha\in\mathbb{R}$, $\mathcal{X}^{(\alpha)}$ represents the element-wise power operator, i.e., $(\mathcal{X}^{(\alpha)})_{i_1\cdots i_n}=\mathcal{X}_{i_1\cdots i_n}^{\alpha}$. $\boldsymbol{1}_{\Omega}\in\mathbb{R}^{d_1\times\cdots\times d_n}$ denotes the tensor with 1 on $\Omega$ and $0$ otherwise. We use $\mathcal{X}\succ 0$ to denote the tensor with $\mathcal{X}_{i_1\cdots i_n}>0$ for all $i_1,\cdots, i_n$. Moreover, we say that $\Omega \sim \mathcal{W}$ if  the entries of $\mathcal{X}$ are sampled randomly with the sampling set $\Omega$ such that $(i_1,\cdots,i_n)\in\Omega$  with probability $\mathcal{W}_{i_1\cdots i_n}$. We include here some basic notions relating to tensors, and refer the reader to e.g., \cite{kolda2009tensor} for a more thorough survey.
 
\begin {definition}[Tensor]
A tensor is a multidimensional array. The dimension of a tensor is called the order (also called the mode). The space of real tensors  of order $n$ and  size $d_1\times\cdots\times d_n$ is denoted as $\mathbb{R}^{d_1\times  \cdots\times d_n}$. 
The elements of a tensor $\mathcal{X}\in\mathbb{R}^{d_1\times \cdots\times d_n}$ are denoted by $\mathcal{X}_{i_1\cdots i_n}$.
\end {definition}
An $n$-order  tensor $\mathcal{X}$ can be matricized in $n$ ways by unfolding it along each of the $n$ modes. The definition for the matricization of a given tensor is stated below.
\begin {definition}[Matricization/unfolding of a tensor] 
The mode-$k$ matricization/unfolding of tensor $\mathcal{X}\in\mathbb{R}^{d_1\times  \cdots\times d_n}$ is the matrix, which is denoted as $\mathcal{X}_{(k)}\in\mathbb{R}^{d_k\times\prod\limits_{j\neq k}d_j}$, whose columns are composed of all the vectors obtained from $\mathcal{X}$ by fixing all indices except for the $k$-th dimension.  The mapping $\mathcal{X}\mapsto \mathcal{X}_{(k)}$ is called the mode-$k$ unfolding operator.
\end {definition}
\begin{Example}
Let $\mathcal{X}\in\mathbb{R}^{3\times 4\times 2}$ with the following frontal slices:
\[X_1=\begin{bmatrix}
1&4&7&10\\
2&5&8&11\\
3&6&9&12
\end{bmatrix} \quad~~~ X_2=\begin{bmatrix}
13&16&19&22\\
14&17&20&23\\
15&18&21&24
\end{bmatrix},
\]
then the three mode-$n$ matricizations  are
\begin{eqnarray*}
\mathcal{X}_{(1)}&=&\begin{bmatrix}
1&4&7&10&13&16&19&22\\
2&5&8&11&14&17&20&23\\
3&6&9&12&15&18&21&24
\end{bmatrix},\\
\mathcal{X}_{(2)}&=&\begin{bmatrix}
1&2&3&13&14&15\\
4&5&6&16&17&18\\
7&8&9&19&20&21\\
10&11&12&22&23&24
\end{bmatrix},\\
\mathcal{X}_{(3)}&=&\begin{bmatrix}
1&2&3&\cdots&10&11&12\\
13&14&15&\cdots&22&23&24
\end{bmatrix}.
\end{eqnarray*}
\end{Example}
\begin{definition}[Folding operator]
Suppose that $\mathcal{X}$ is a tensor. The mode-$k$ folding operator of a matrix $M=\mathcal{X}_{(k)}$, denoted as $\text{fold}_k(M)$, is the inverse operator of the unfolding operator.
\end{definition}
\begin{definition}[$\infty$-norm]
Given $\mathcal{X}\in\mathbb{R}^{d_1\times \cdots\times d_n}$, the norm $\left\|\mathcal{X}\right\|_\infty$ is defined as
\[\left\|\mathcal{X}\right\|_{\infty}=\max_{i_1,\cdots,i_n}|\mathcal{X}_{i_1\cdots i_n}|.
\] 

The unit ball under the $\infty$-norm is denoted by {$\boldsymbol{B}$}$_\infty$. 
\end{definition}
\begin{definition}[Frobenius norm]
The Frobenius norm for a tensor $\mathcal{X}\in\mathbb{R}^{d_1\times  \cdots\times d_n}$ is defined as
\[\left\|\mathcal{X}\right\|_F=\sqrt{\sum_{i_1,\cdots, i_n}\mathcal{X}_{i_1\cdots i_n}^2}.\]
\end{definition}
\begin{definition}[Max-norm for matrix]Given $X\in\mathbb{R}^{d_1\times d_2}$, the max-norm for $X$ is defined as
\[\left\|X\right\|_{max}=\min_{X=UV^T}\left\|U\right\|_{2,\infty}\left\|V\right\|_{2,\infty}.
\]
\end{definition}
\begin{definition}[Product operations]~
\begin{enumerate}
\item[$\bullet$]Outer product: Let $\boldsymbol{a_1}\in\mathbb{R}^{d_1}, \cdots, \boldsymbol{a_n}\in\mathbb{R}^{d_n}$. The outer product among these $n$ vectors is a tensor $\mathcal{X}\in\mathbb{R}^{d_1\times\cdots\times d_n}$ defined as:
\[\mathcal{X}=\boldsymbol{a}_1\out \cdots\out \boldsymbol{a}_n, ~~ \mathcal{X}_{i_1,\cdots,i_n}=\prod\limits_{k=1}^n \boldsymbol{a}_k(i_k).\]
The tensor $\mathcal{X}\in\mathbb{R}^{d_1\times\cdots\times d_n}$ is of rank one if it can be written as the outer product of $n$ vectors.
    \item [$\bullet$]Kronecker product of matrices: The Kronecker product of $A\in\mathbb{R}^{I\times J}$ and $B\in\mathbb{R}^{K\times L}$ is denoted by $A\kron B$. The result is a matrix of size $(KI)\times (JL)$ defined by
\begin{eqnarray*}
A\kron B&=&\begin{bmatrix}
A_{11}B&A_{12}B&\cdots &A_{1J}B\\
A_{21}B&A_{22}B&\cdots &A_{2J}B\\
\vdots&\vdots&\ddots&\vdots\\
A_{I1}B&A_{I2}B&\cdots&A_{IJ}B
\end{bmatrix}.
\end{eqnarray*}
\item[$\bullet$] Khatri-Rao product:
Given matrices $A\in\mathbb{R}^{d_1\times r}$ and $B\in\mathbb{R}^{d_2\times r}$, their Khatri-Rao product is denoted by $A\khatri B$. The result is a matrix of size  $(d_1d_2)\times r$ defined by
\[A\khatri B=\begin{bmatrix}
\boldsymbol{a_1}\kron \boldsymbol{b_1}&\cdots&\boldsymbol{a_r}\kron\boldsymbol{ b_r}
\end{bmatrix},
\]
where $\boldsymbol{a_i}$ and $\boldsymbol{b_i}$ stand for the $i$-th column of $A$ and $B$ respectively. 
\item [$\bullet$] Hadamard product: 
Given $\mathcal{X},\mathcal{Y}\in\mathbb{R}^{d_1\times\cdots\times d_n}$, their Hadamard product $\mathcal{X}\hadam \mathcal{Y} \in\mathbb{R}^{d_1\times\cdots\times d_n}$ is defined by element-wise multiplication, i.e.,
\[(\mathcal{X}\hadam\mathcal{Y})_{i_1\cdots i_n}=\mathcal{X}_{i_1\cdots i_n}\mathcal{Y}_{i_1\cdots i_n}.
\]
\item[$\bullet$] Mode-$k$ product: Let $\mathcal{X}\in\mathbb{R}^{d_1\times \cdots\times d_n}$ and $U\in\mathbb{R}^{d_k\times J}$, the multiplication between $\mathcal{X}$ on its mode-$k$ with $U$ is denoted as $\mathcal{Y}=\mathcal{X}\times_k U$ with 
\[\mathcal{Y}_{i_1,\cdots,i_{k-1},j,i_{k+1},\cdots,i_{n}}=\sum_{s=1}^{d_k}\mathcal{X}_{i_1,\cdots,i_{k-1},s,i_{k+1},\cdots,i_{n}}U_{s,j}.
\]
\end{enumerate}
\end{definition}
\begin{definition}[Tensor (CP) rank \cite{HF1927,HF1928}]
The (CP) \textit{rank} of a tensor $\mathcal{X}$, denoted $\text{rank}(\mathcal{X})$, is defined as the smallest number of rank-1 tensors that generate $\mathcal{X}$ as their sum. We use $K_r$ to denote the cone of rank-r tensors.
\end{definition}
 Given  ${}^kM\in\mathbb{R}^{d_k\times r}$, we use $\llbracket {}^1M,\cdots,{}^nM\rrbracket$ to denote the CP representation of tensor \mbox{$\mathcal{X}$, i.e.,}
\[\mathcal{X}=\sum_{j=1}^r \left({}^1M(:,j)\out\cdots\out{}^nM(:,j)\right),
\]
where $M(:,j)$ means the $j$-th column of the matrix $M$.

Different from the case of matrices, the rank of a tensor is not presently  well understood. Additionally, the task of computing the rank of a tensor is an NP-hard problem \cite{k1989}.  Next we  introduce an alternative definition of the rank of a tensor, which is easy to compute.
\begin{definition}[Tensor Tucker rank  \cite{HF1928}]Let $\mathcal{X}\in\mathbb{R}^{d_1\times\cdots\times d_n}$.
The tuple $(r_1,\cdots,r_n)\in\mathbb{N}^{n}$ is called the Tucker rank of the tensor $\mathcal{X}$, where $r_k=\text{rank}(\mathcal{X}_{(k)})$. We use $K_{\boldsymbol{r}}$ to denote the cone of  tensors with   Tucker rank {$\boldsymbol{r}$}. 
\end{definition}
 
Tensor decompositions  are powerful tools for extracting meaningful, latent structures
in heterogeneous, multidimensional data (see e.g., \cite{kolda2009tensor}). In this paper,  we focus on two most widely used decomposition methods: CP and HOSVD. For  more comprehensive introduction, readers  are referred to \cite{AY2008,kolda2009tensor,SDFHPF2017}. 

\subsubsection{CP-Based Method for Tensor Completion} 
The CP decomposition was first proposed by Hitchcock \cite{HF1927} and further discussed \mbox{in  \cite{CC1970}}. The formal definition of the  CP decomposition is  the following.
\begin{definition}[CP decomposition]
Given a tensor $\mathcal{X}\in\mathbb{R}^{d_1\times\cdots\times d_n}$, its CP decomposition  is an approximation of $n$ loading matrices $A_{k}\in\mathbb{R}^{d_k\times r}$, $k=1,\cdots,n$, such that
\[ \mathcal{X}\approx\llbracket A_1,\cdots,A_n\rrbracket=\sum_{i=1}^{r}A_1(:,i)\out\cdots\out A_{n}(:,i),
\]
where $r$ is a positive integer denoting an upper bound of the rank of $\mathcal{X}$ and $A_k(:,i)$ is the $i$-th column of matrix $A_k$. If we  unfold $\mathcal{X}$ along its $k$-th mode, we have 
\[\mathcal{X}_{(k)}\approx A_{k}(A_1\khatri\ldots\khatri A_{k-1}\khatri A_{k+1}\khatri\cdots\khatri A_{n})^T.
\]
\end{definition}
 {Here the $\approx$ sign means that the algorithm should find an optimal $\widehat{\tens{X}}$ with the given rank such that the distance between the low-rank approximation and the original tensor, $\|\tens{X}-\widehat{\tens{X}}\|_F$, is minimized.}

Given an observation set $\Omega$, the main idea to implement   tensor completion for a low-rank tensor $\mathcal{T}$ is to conduct imputation based on the equation 
\[ \mathcal{X}= \mathcal{T}_{\Omega}+\widehat{\mathcal{X}}_{\Omega^c},
\] where $\widehat{\mathcal{X}}=\llbracket A_1,\cdots,A_n\rrbracket$ is the interim low-rank approximation based on  the CP decomposition, $\mathcal{X}$ is the recovered tensor used in next iteration for decomposition, and $\Omega^c=\left\{(i_1,\cdots,i_n):1\leq i_k\leq d_k \right\}\setminus\Omega$. For each iteration, we  usually 
estimate the matrices $A_k$ using the alternating least squares optimization method (see e.g., \cite{BR1997,KTB1999,TB2005}).
\vspace{0.1pt}
\subsubsection{HOSVD-Based Method for Tensor Completion} 
The Tucker decomposition was proposed by Tucker \cite{tucker1966} and further developed \mbox{in \cite{DDV2000,KD1980}.}
\begin{definition}[Tucker decomposition]
Given an $n$-order tensor $\mathcal{X}$, its Tucker decomposition is defined as an approximation of a  core tensor $\mathcal{C}\in\mathbb{R}^{r_1\times\cdots \times r_n}$ multiplied by $n$ factor matrices $A_k\in\mathbb{R}^{d_k\times r_k}$
, $k=1,\cdots,n$ along each mode, such that
\[\mathcal{X}\approx\mathcal{C}\times_1 A_1\times_2\cdots\times_n A_n=\llbracket \mathcal{C};A_1,\cdots,A_n\rrbracket,
\] where $r_k$ is a positive integer denoting an upper bound of the rank of the matrix $\mathcal{X}_{(k)}$.

If we unfold $\mathcal{X}$ along its $k$-th mode, we have
\[\mathcal{X}_{(k)}\approx A_{k}\mathcal{C}_{(k)}(A_1\kron \cdots \kron A_{k-1}\kron A_{k+1}\kron\cdots\kron A_{n})^T
\]
\end{definition}
Tucker decomposition is  a widely used tool for tensor completion. 
To implement  Tucker decomposition, one popular method is called the higher-order SVD \mbox{(HOSVD)  \cite{tucker1966}}.  The main idea of HOSVD is:
\begin{enumerate}
    \item Unfold $\mathcal{X}$ along mode $k$ to obtain matrix $\mathcal{X}_{(k)}$;
    \item Find the economic SVD decomposition of $\mathcal{X}_{(k)}={}^kU{}^k\Sigma {}^kV^T$;
    \item  Set $A_k$ to be the first $r_k$ columns of ${}^kU$;
    \item $\mathcal{C}=\mathcal{X}\times_1 A_1^T\times_2\cdots\times_n A_n^T$.
\end{enumerate}

If we want to find a Tucker rank $\boldsymbol{r}=[
r_1,\cdots,r_n]$ approximation for the tensor  $\mathcal{X}$ via HOSVD process, we just replace $A_k$ by the first $r_k$ columns of $U_k$.

\subsubsection{Tensor Completion Problem under Study}
In our setting, it is supposed that $\mathcal{T}$ is an unknown tensor  in $K_r\cap\beta \boldsymbol{B}_\infty$ or  $K_{\boldsymbol{r}}\cap \beta \boldsymbol{B}_\infty$. Fix a sampling pattern $\Omega\subseteq[d_1]\times\cdots\times[d_n]$ and the weight tensor $\mathcal{W}$.
Our goal is  to design an algorithm that gives provable guarantees for a worst-case $\mathcal{T}$, even if it is adapted to $\Omega$. 

In our algorithm, the observed data are $\mathcal{T}_\Omega+\mathcal{Z}_\Omega=\boldsymbol{1}_{\Omega}\hadam \left(\mathcal{T}+\mathcal{Z}\right)$, where $\mathcal{Z}_{i_1\cdots i_n}\sim\mathcal{N}(0,\sigma^2)$ are i.i.d. Gaussian random variables. From the observations, the goal is to learn something about $\mathcal{T}$. In this paper, instead of measuring our recovered results with the underlying true tensor in a standard Frobenius norm $\|\mathcal{T}-\widehat{\mathcal{T}}\|_F$, we are interested in learning $\mathcal{T}$ using a \textit{weighted} Frobenius norm, i.e.,  to develop an efficient algorithm to find $\widehat{\mathcal{T}}$ so that 
\begin{equation*}
\left\|\mathcal{W}^{(1/2)}\hadam(\mathcal{T}-\widehat{\mathcal{T}})\right\|_F
\end{equation*} is as small as possible for some weight tensor $\mathcal{W}$. When measuring the weighted error, it is important to normalize appropriately to understand the meaning of the error bounds. In our results, we   always normalize the error bounds by $\left\|\mathcal{W}^{(1/2)}\right\|_F$. It is noteworthy that
\begin{eqnarray*}
\frac{\left\|\mathcal{W}^{(1/2)}\hadam(\mathcal{T}-\widehat{\mathcal{T}})\right\|_F}{\left\|\mathcal{W}^{(1/2)}\right\|_F}&=&\left( \sum_{i_1,\cdots,i_n}\frac{\mathcal{W}_{i_1\cdots i_n}}{\sum_{i_1,\cdots,i_n}\mathcal{W}_{i_1,\cdots,i_n}}(\mathcal{T}_{i_1\cdots i_n}-\widehat{\mathcal{T}}_{i_1\cdots i_n})^2\right)^{1/2},
\end{eqnarray*}
which gives a weighted average of the per entry squared error.  Generally, our problem can be formally stated below.



	
	
	
	
	
	
	
	
	
\IncMargin{1em}
\begin{algorithm}[H]\label{algorithm iterative1}
	\SetKwInOut{Problem}{Problem}
	
	\Problem{Weighted Universal Tensor Completion}
	{\bf Parameters:}
	
	{$\bullet$ Dimensions $d_1, \cdots, d_n$;}
	
	{$\bullet$ A sampling pattern $\Omega\subseteq [d_1]\times\cdots\times[d_n]$;}
	
	{$\bullet$ Parameters $\sigma, \beta>0$, $r$ or $\boldsymbol{r}=[r_1~\cdots~r_n]$;}
	
	{ $\bullet$ A rank-1 weight tensor $\mathcal{W}\in\mathbb{R}^{d_1\times \cdots\times d_n}$ so that $\mathcal{W}_{i_1\cdots i_n}>0$ for all $i_1,\cdots,i_n$;}
	
	{$\bullet$ A set $K$ (e.g., $K_r\cap\beta \boldsymbol{B}_\infty$ or $K_{\boldsymbol{r}}\cap\beta \boldsymbol{B}_\infty$).}
	
{{\bf Goal:} Design an efficient algorithm $\mathcal{A}$ with the following guarantees:}
	
	{$\bullet$ $\mathcal{A}$ takes as input entries $\mathcal{T}_\Omega+\mathcal{Z}_\Omega$ so that $\mathcal{Z}_{i_1\cdots i_n}\sim \mathcal{N}(0,\sigma^2)$ are i.i.d.;}
	
	{$\bullet$ $\mathcal{A}$ runs in polynomial time;}
	
	{$\bullet$ With high probability over the choice of $\mathcal{Z}$, $\mathcal{A}$ returns an estimate $\widehat{\mathcal{T}}$ of $\mathcal{T}$ so that
	\[
	\frac{\left\|\mathcal{W}^{(1/2)}\hadam(\mathcal{T}-\widehat{\mathcal{T}}) \right\|_F}{\left\|\mathcal{W}^{(1/2)}\right\|_F}\leq\delta\]
	for all $\mathcal{T}\in K$, where $\delta$ depends on the problem parameters.}
\end{algorithm}

 \begin{Remark}[Strictly positive $\mathcal{W}$] The requirement that $\mathcal{W}_{i_1\cdots i_n}$ is strictly greater than zero is a generic condition. In fact, if $\mathcal{W}_{i_1\cdots i_n}=0$ for some $(i_1,\cdots,i_n)$,   some mode $k$ with index $i_k$ of $\mathcal{W}$ is zero,  then we can reduce the problem to a smaller one by ignoring that mode $k$ with index $i_k$. 
 \end{Remark}
\section{Main  Results}\label{sec:results}
In this section, we state informal versions of our main results.  
With  fixed sampling pattern $\Omega$ and weight tensor $\mathcal{W}$, we can find  $\widehat{\mathcal{T}}$ by solving the following \mbox{optimization problem}:
\begin{equation}\label{eqn:obj}
    \widehat{\mathcal{T}}=\mathcal{W}^{(-1/2)}\hadam\underset{{\rank(\mathcal{X})=r}}{\argmin}\left\|\mathcal{X}-\mathcal{W}^{(-1/2)}\hadam\mathcal{Y}_\Omega \right\|_F,
\end{equation}
or
\begin{equation}\label{eqn:obj1}
    \widehat{\mathcal{T}}=\mathcal{W}^{(-1/2)}\hadam\underset{{\text{Tucker-}\rank(\mathcal{X})=\boldsymbol{r}}}{\argmin}\left\|\mathcal{X}-\mathcal{W}^{(-1/2)}\hadam\mathcal{Y}_\Omega \right\|_F,
\end{equation}
where $\mathcal{Y}_\Omega\in\mathbb{R}^{d_1\times\cdots\times d_n}$ with $$\mathcal{Y}_\Omega(i_1,\cdots,i_n)=\begin{cases}
    \mathcal{T}_{i_1\cdots i_n}+\mathcal{Z}_{i_1\cdots i_n}      & \quad \text{if } (i_1,\cdots,i_n)\in\Omega \\
   0 & \quad \text{if } (i_1,\cdots,i_n)\not\in\Omega
  \end{cases}
.$$ 

It is known that solving  \eqref{eqn:obj} is NP-hard. However, there are some polynomial time algorithms   
 {to find approximate solutions for \eqref{eqn:obj} such that the approximation is (empirically) close to the actual solution of \eqref{eqn:obj} in terms of the Frobenius norm}.
In our numerical experiments, we solve $\eqref{eqn:obj}$ via the   CP-ALS algorithm \cite{CC1970}.
To solve \eqref{eqn:obj1}, we use  the HOSVD process \cite{DDV2000}. Assume that $\mathcal{T}$ has Tucker rank $\boldsymbol{r}=[r_1,\cdots,r_n]$.  Let
\[\widehat{A}_i=\underset{\rank(A)=r_i}{\argmin}\left\|A-(\mathcal{W}^{(-1/2)}\hadam \mathcal{Y}_\Omega)_{(i)}\right\|_2
\]
and set $\widehat{U}_i$ to be the left singular vector matrix of $\widehat{A}_i$. Then the estimated tensor is of \mbox{the form} \[\widehat{\mathcal{T}}=\mathcal{W}^{(-1/2)}\hadam((\mathcal{W}^{(-1/2)}\hadam\mathcal{Y}_\Omega)\times_1 \widehat{U}_1\widehat{U}_1^T\times_2 \cdots\times_n \widehat{U}_n\widehat{U}_n^T.\]

 In the following, we call this the weighted HOSVD algorithm.
\subsection{General Upper Bound} 
Suppose that the optimal solution $\widehat{\mathcal{T}}$ for \eqref{eqn:obj} or \eqref{eqn:obj1} $\widehat{\mathcal{T}}$ can be found, we would like to give an upper bound estimations for $\|\mathcal{W}^{(1/2)}\hadam(\mathcal{T}-\widehat{\mathcal{T}})\|_F$ with some proper weight tensor $\mathcal{W}$.
\begin{theorem}\label{thm:gub tensor}
Let $\mathcal{W}=\boldsymbol{w}_1\out\cdots\out\boldsymbol{w}_n \in\mathbb{R}^{d_1\times \cdots\times d_n}$ have strictly positive entries, and fix $\Omega\subseteq[d_1]\times\cdots\times[d_n]$. Suppose that $\mathcal{T}\in\mathbb{R}^{d_1\times \cdots\times d_n}$ has rank $r$ for problem \eqref{eqn:obj} or Tucker  rank $\boldsymbol{r}=[r_1, \cdots,r_n]$ for problem \eqref{eqn:obj1}, and let $\widehat{\mathcal{T}}$ be the optimal solutions for \eqref{eqn:obj} or \eqref{eqn:obj1}. Suppose that $\mathcal{Z}_{i_1\cdots i_n}\sim\mathcal{N}(0,\sigma^2)$.
Then with 
probability at least $1-2^{-|\Omega|/2}$ 
over the choice of $\mathcal{Z}$,
\[\left\|\mathcal{W}^{(1/2)}\hadam(\mathcal{T}-\widehat{\mathcal{T}})\right\|_F\leq 2\left\|\mathcal{T}\right\|_{\infty}\left\|\mathcal{W}^{(1/2)}-\mathcal{W}^{(-1/2)}\hadam \boldsymbol{1}_{\Omega}\right\|_F+4\sigma\mu\sqrt{|\Omega|\log(2)},
\]
 {Recall here, $(\mathcal{W}^{(1/2)})_{i_1\cdots i_n} =\mathcal{W}_{i_1\cdots i_n}^{(1/2)}$ and $(\mathcal{W}^{(-1/2)})_{i_1\cdots i_n} =\mathcal{W}_{i_1\cdots i_n}^{(-1/2)}$as defined in \mbox{Section \ref{sec: notation}} and } $\mu^2=\max_{(i_1,\cdots,i_n)\in\Omega}\frac{1}{\mathcal{W}_{i_1\cdots i_n}}$.
\end{theorem}
Notice that the upper bound in Theorem \ref{thm:gub tensor} is for the optimal output $\widehat{\mathcal{T}}$ for problems \eqref{eqn:obj} and \eqref{eqn:obj1}, which is general. However, the upper bound in Theorem \ref{thm:gub tensor} contains no   rank information of the underlying tensor $\mathcal{T}$. To introduce the rank information of the underlying tensor $\mathcal{T}$, we restrict our analysis for Problem \eqref{eqn:obj1} by considering the  HOSVD process in \mbox{the sequel.}

\subsection{Results for Weighted HOSVD Algorithm} 
In this section, we begin by giving a general upper bound  for the weighted\linebreak \mbox{HOSVD algorithm}.
\subsubsection{General Upper Bound for Weighted HOSVD}
\begin{theorem}[Informal, see Theorem \ref{thm:ub low-rank tensor}]\label{thm: gubwHOSVD}
Let $\mathcal{W}=\boldsymbol{w_1}\out\cdots\out \boldsymbol{w_n}\in\mathbb{R}^{d_1\times\cdots \times d_n}$ have strictly positive entries, and fix $\Omega\subseteq[d_1]\times\cdots\times[d_n]$. Suppose that $\mathcal{T}\in\mathbb{R}^{d_1\times \cdots\times d_n}$ has Tucker rank $\textbf{r}=[r_1,\cdots,r_n]$. Suppose that $\mathcal{Z}_{i_1\cdots i_n}\sim\mathcal{N}(0,\sigma^2)$ and let $\widehat{\mathcal{T}}$ be the estimate of the solution of   \eqref{eqn:obj1} via
the HOSVD   process.
Then 
\begin{multline*}\label{uppb}
    \left\|\mathcal{W}^{(1/2)}\hadam(\mathcal{T}-\widehat{\mathcal{T}})\right\|_F\lesssim \left(\sum_{k=1}^n \sqrt{r_k\log(d_k+\prod\limits_{j\neq k}d_j)}\mu_k\right)\sigma\\
   +\left(\sum_{k=1}^{n}r_k\left\|(\mathcal{W}^{(-1/2)}\hadam \boldsymbol{1}_{\Omega}-\mathcal{W}^{(1/2)})_{(k)}\right\|_2 \right)\left\|\mathcal{T}\right\|_\infty,
\end{multline*} with high probability over the choice of $\mathcal{Z}$,
where 
 \[\mu_k^2=\max\left\{\max_{i_k}\left(\sum_{i_1,\cdots,i_{k-1},i_{k+1},\cdots,i_n}\frac{1_{(i_1,i_2,\cdots,i_n)\in\Omega}}{\mathcal{W}_{i_1i_2\cdots i_{n}}}\right),\max_{i_1,\cdots,i_{k-1},i_{k+1},\cdots,i_n}\left(\sum_{i_k}\frac{1_{(i_1,i_2,\cdots,i_n)\in\Omega}}{\mathcal{W}_{i_1i_2\cdots i_{n}}}\right) \right\}.\]
and $a\lesssim b$   means that $a\leq cb$ for some universal constant $c>0$.
\end{theorem}

\begin{Remark}
 {The upper bound in \cite{foucart2019weighted} suggests $\|W^{(1/2)}\hadam(M-\widehat{M})\|_F\leq 2\sqrt{2}r\lambda\|M\|_{\infty}$\linebreak$+4\sqrt{2}\sigma\mu_1\sqrt{r\log(d_1+d_2)} $, where $\lambda = \|W^{(1/2)} - W^{(-1/2)} \circ \mathbf{1}_{\Omega}\|$ and $\mu_1^2 = max_{(i,j)\in\Omega} \frac{1}{W_{ij}}$, where $\hat{M}$ is obtained by considering the truncated SVD decompositions. Notice that in our result, when $n=2$, the upper bound becomes $2\sqrt{r\log(d_1+d_2)}\mu\sigma + 2r\|W^{(1/2)}-W^{(-1/2)}\circ\mathbf{1}_{\Omega}\|\|M\|_{\infty}$ with $\mu^2=\max\{\|\mathbf{1}_{\Omega}\circ W^{(-1)}\|_{\infty},\|\mathbf{1}_{\Omega}\circ W^{(-1)}\|_{1}\}$. Since  $\mu$ in our work is much bigger than the $\mu_1$ in \cite{foucart2019weighted}, the bound in our work is weaker than the one in \cite{foucart2019weighted}. The reason is that in order to obtain a general bound for all tensor, the fact that  the optimal approximations $\hat{M}$ for a given matrix in the spectral norm and Frobenious norm are the same cannot be applied.
}
\end{Remark}

\subsubsection{Case Study: When \texorpdfstring{$\Omega\sim\mathcal{W}$}{Lg}} 
 To understand the bounds mentioned above, we also study the case when $\Omega\sim\mathcal{W}$ such that $\|(\mathcal{W}^{(1/2)}-\mathcal{W}^{(-1/2)}\hadam \boldsymbol{1}_{\Omega})_{(k)}\|_2$ is small for $k=1,\cdots,n$. Even though the samples are taken randomly in this case, our goal is to understand our upper bounds for deterministic sampling pattern $\Omega$. To make sure that $\|(\mathcal{W}^{(1/2)}-\mathcal{W}^{(-1/2)}\hadam \boldsymbol{1}_{\Omega})_{(k)}\|_2$ is  small, we need to assume that each entry of $\mathcal{W}$ is not too small.  For this case, we have the following main results.
 
 \begin{theorem}[Informal, see Theorems \ref{thm: omgsw} and \ref{thm: lb}]\label{thm:bfsomega}
Let $\mathcal W=\boldsymbol{w_1}\out \cdots \out \boldsymbol{w_n}\in\mathbb{R}^{d_1\times \cdots\times d_n}$ be a CP rank-1 tensor so that for all $(i_1,\cdots,i_n)\in[d_1]\times\cdots\times[d_n]$ we have $\mathcal W_{i_1\cdots i_n}\in[\frac{1}{\sqrt{d_1\cdots d_n}},1]$. Suppose that $\Omega\sim\mathcal{W}$. 
\begin{enumerate}
    \item [$\bullet$]\textbf{Upper bound:}
Then the following holds with high probability.\\
For our weighted HOSVD algorithm $\mathcal{A}$, for any Tucker $\text{rank-}\boldsymbol{r}$ tensor $\mathcal{T}$ with $\left\|\mathcal{T}\right\|_\infty\leq\beta$, $\mathcal{A}$ returns $\widehat{\mathcal{T}}=\mathcal{A}(\mathcal{T}_\Omega+\mathcal{Z}_\Omega)$ so that with high probability 
over the choice of $\mathcal{Z}$,
\begin{align*}
\frac{\left\|\mathcal{W}^{(1/2)}\hadam(\mathcal{T}-\widehat{\mathcal{T}})\right\|_F}{\left\|\mathcal{W}^{(1/2)}\right\|_F}
   &\lesssim\frac{1}{\sqrt{|\Omega|}}\left(\beta n^2rd^{\frac{n-1}{2}}\log(d)+\sigma n^2r^{1/2}d^{\frac{n-1}{2}} \right),
\end{align*}
where $r=\max_{k}\{r_{k}\}$ and $d=\max_{k}\{d_k\}$.
\item [$\bullet$] \textbf{Lower bound: } If additionally, $\mathcal{W}$ is flat (the entries of $\mathcal{W}$ are close), then for our weighted HOSVD algorithm $\mathcal{A}$, there exists some $\mathcal{T}\in K_{\boldsymbol{r}}\cap\beta\mathbf{B}_{\infty}$  so that with probability at least $\frac{1}{2}$ over the choice of $\mathcal{Z}$, \begin{eqnarray*}
&&\frac{\left\|\mathcal{W}^{(1/2)}\hadam(\mathcal{A}(\mathcal{T}_\Omega+\mathcal{Z}_\Omega)-\mathcal{T})\right\|_F}{\left\|\mathcal{W}^{(1/2)}\right\|_F}\\
&\gtrsim& \min\left\{\frac{\sigma}{\sqrt{|\Omega|}}\left(\frac{\tilde{r}\tilde{d}}{\tilde{d}+2C'^2\tilde{r}}\right)^{\frac{n}{2}},\frac{\sigma}{\sqrt{|\Omega|}}\left(\frac{\tilde{r}\tilde{d}}{\left(\sqrt{\tilde{d}}+\sqrt{2\tilde{r}\log(\tilde{r})}C'\right)^2}\right)^{\frac{n}{2}},\frac{\beta}{\sqrt{n\log(\tilde{d})}}\right\},
\end{eqnarray*}
where  $\tilde{r}=\min_{k}\{r_k\}$,  $\tilde{d}=\min_{k}\{d_k\}$, and $C'$ is some constant to measure the ``flatness" \mbox{of $\mathcal{W}$}.
\end{enumerate}
\end{theorem}
 \begin{Remark}
 The formal statements in Theorems  \ref{thm: omgsw} and \ref{thm: lb} are more general than the statements in Theorem \ref{thm:bfsomega}.
 \end{Remark}

\section{Experiments}\label{section:simulations}
\subsection{Simulations for Uniform Sampling Pattern } \label{sim: uniform}
In this section, we test the performance of our weighted HOSVD algorithm when the sampling pattern arises from uniform random sampling. Consider a tensor $\mathcal{T}$ of the form $\mathcal{T}=\mathcal{C}\times_1 U_1\times_2 \cdots\times_n U_n$, where $U_i\in\mathbb{R}^{d_i\times r_i}$ and $\mathcal{C}\in\mathbb{R}^{r_1\times \cdots \times r_n}$. Let $\mathcal{Z}$ be a  Gaussian random tensor with $\mathcal{Z}_{i_1\cdots i_n}\in\mathcal{N}(0,\sigma)$ and $\Omega$ be the sampling pattern set  according to uniform sampling.
In this simulation, we  compare the results of numerical experiments for using the HOSVD algorithm to solve 
\begin{equation}\label{eqn:rgl_obj}
\widehat{\mathcal{T}}=\underset{\text{Tucker}\_\rank(\mathcal{X})=\boldsymbol{r}}{\argmin}\left\|\mathcal{X}-\mathcal{Y}_{\Omega}\right\|_F,
\end{equation} 
\begin{equation}\label{eqn:std_obj}
\widehat{\mathcal{T}}=\underset{\text{Tucker}\_\rank(\mathcal{X})=\boldsymbol{r}}{\argmin}\left\|\mathcal{X}-\frac{1}{p} \mathcal{Y}_{\Omega}\right\|_F,
\end{equation} and  
\begin{equation}\label{eqn:bia_obj}
\widehat{\mathcal{T}}=\mathcal W^{(-1/2)}\hadam\underset{\text{Tucker}\_\rank(\mathcal{X})=\boldsymbol{r}}{\argmin} \left\|\mathcal{X}-\mathcal W^{(-1/2)}\hadam \mathcal{Y}_{\Omega}\right\|_F,
\end{equation}
 where $p=\frac{|\Omega|}{\prod_{k=1}^{n}d_k}$ and $\mathcal{Y}_\Omega=\mathcal{T}_\Omega+\mathcal{Z}_\Omega$.

First, we generate a synthetic sampling set $\Omega$ with sampling rate SR$:\,=\frac{|\Omega|}{\prod_{k=1}^{n}d_k}=30\%$ and find a weight tensor $\mathcal{W}$   by solving 
\begin{equation}
    \mathcal{W}=\underset{\mathcal{X}\succ 0,\text{rank}(\mathcal{X})=1}{\argmin}\|\mathcal{X}-\boldsymbol{1}_{\Omega}\|_F
\end{equation}
 via the alternating least squares method for the non-negative CP decomposition. Next, we generate
synthetic  tensors $\mathcal{T}\in\mathbb{R}^{d_1\times\cdots\times d_n}$  of the form
$\mathcal{C}\times_1 U_1\times_2 \cdots\times_n U_n$ with $n=3,4$ with $\rank(\mathcal{T}_{(i)})=r$, where $i=1,\cdots,n$, and $r$ varies from $2$ to $10$.
Then we add mean zero Gaussion random noise $\mathcal{Z}$ with variance $\sigma=10^{-2}$ so that a new tensor is generated, which is denoted by $\mathcal{Y}=\mathcal{T}+\mathcal{Z}$.  Then we solve the tensor completion problems \eqref{eqn:rgl_obj}, \eqref{eqn:std_obj} and \eqref{eqn:bia_obj} by the HOSVD procedure. For each fixed low-rank tensor, we average over $20$ tests. We measure error using the weighted Frobenius norm.  The simulation results are reported in Figures \ref{fig:T3_rk_error_Uniform} and \ref{fig:T4_rk_error_Uniform}.  Figure \ref{fig:T3_rk_error_Uniform} shows the results for the tensor of size $100\times 100\times 100$ and \mbox{Figure \ref{fig:T4_rk_error_Uniform}} shows the results for the tensor of size $50\times 50\times 30\times 30$, where the weighted error 
is of the form $\frac{\|\mathcal{W}^{(1/2)}\hadam(\widehat{\mathcal{T}}-\mathcal{T})\|_F}{\|\mathcal{W}^{(1/2)}\|}$.
These figures demonstrate that using our  weighted samples performs more efficiently than using the original samples. For the uniform sampling case, the $p$ weighted samples and $\mathcal{W}$ weighted samples exhibit similar performance.
\vspace{-12pt}
\begin{figure}[H]
   \centering
		\includegraphics[width=10cm]
		{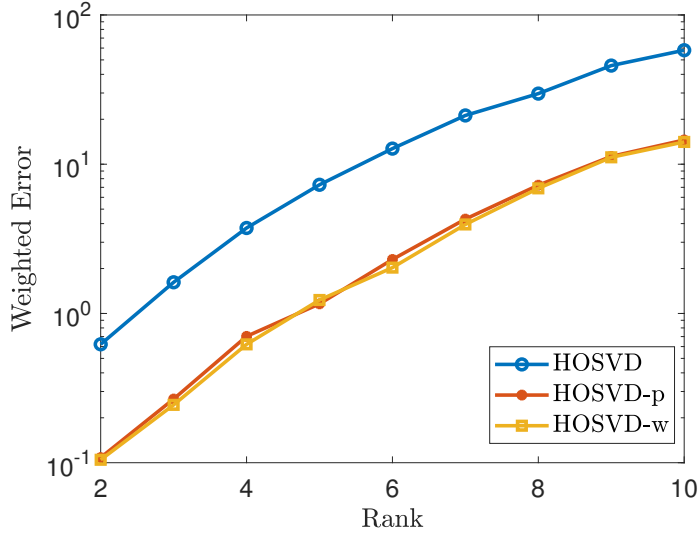}
		
		\label{FIG:T3_Weighted_error0}
	  \caption{Tensor of size $100\times 100\times 100$ using the uniform sampling pattern:
	  plots the errors of the form $\frac{\|\mathcal{W}^{(1/2)}\hadam(\widehat{\mathcal{T}}-\mathcal{T})\|_F}{\|\mathcal{W}^{(1/2)}\|_F}$. 
	  The lines labeled as HOSVD, HOSVD-p and HOSVD-w represent the results for solving \eqref{eqn:rgl_obj}, \eqref{eqn:std_obj} and \eqref{eqn:bia_obj}, respectively. }
    \label{fig:T3_rk_error_Uniform}
	\end{figure}
	\vspace{-15pt}
	\begin{figure}[H]
	\centering
		\includegraphics[width=10cm]
		{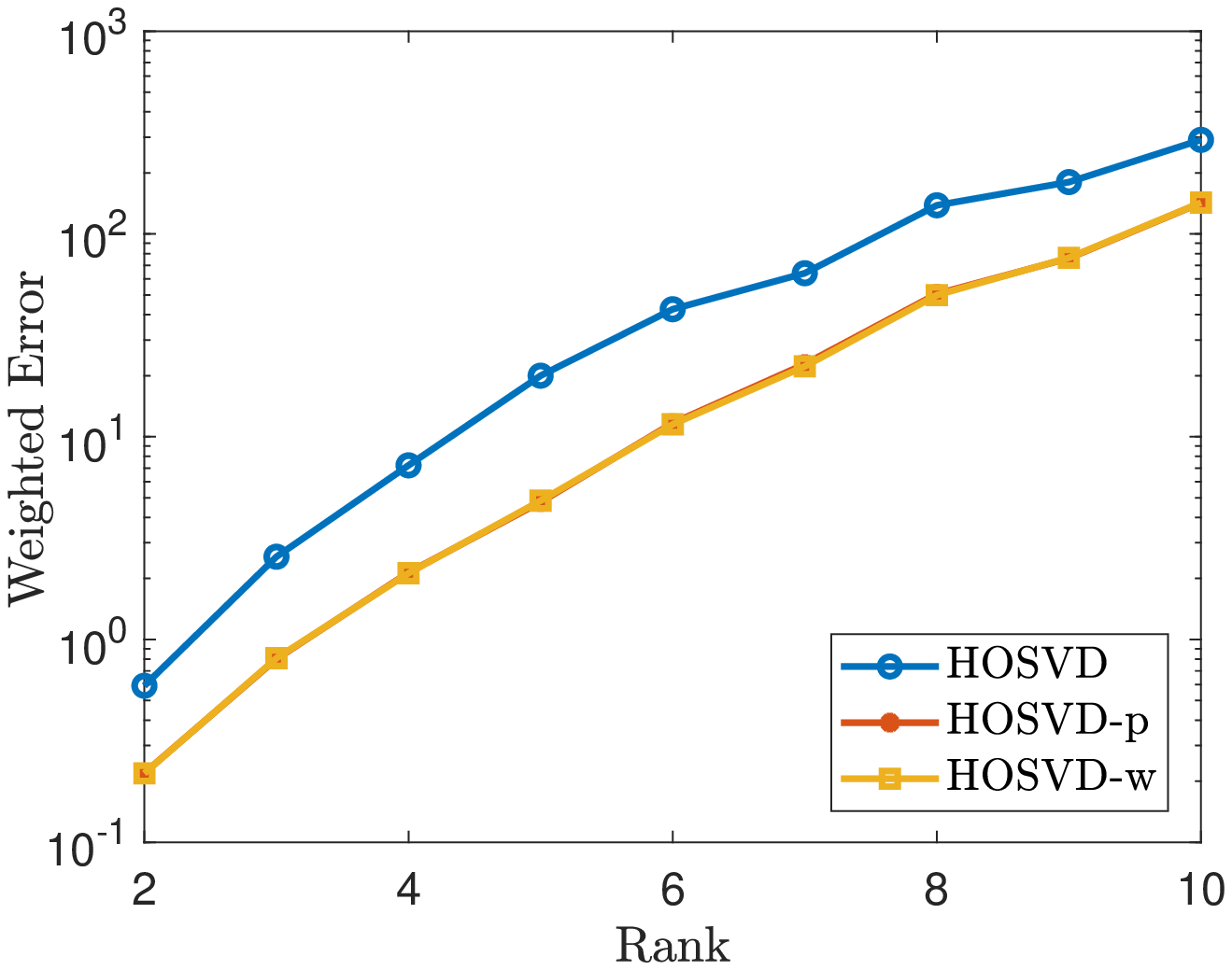}
		\label{FIG:T3_Weighted_error1}
	  \caption{Tensor of size $50\times 50\times 30\times 30$ using the uniform sampling pattern:
	  plots the errors of the form $\frac{\|\mathcal{W}^{(1/2)}\hadam(\widehat{\mathcal{T}}-\mathcal{T})\|_F}{\|\mathcal{W}^{(1/2)}\|_F}$. 
	  The lines labeled as HOSVD, HOSVD-p and HOSVD-w represent the results for solving \eqref{eqn:rgl_obj}, \eqref{eqn:std_obj} and \eqref{eqn:bia_obj}, respectively. }
    \label{fig:T4_rk_error_Uniform}
	\end{figure}
\subsection{Simulation for Non-Uniform Sampling Pattern }
To generate a non-uniform sampling pattern with sampling rate $30\%$, we first  generate a CP rank 1 tensor of the form $\mathcal{H}=\llbracket \boldsymbol{1};\boldsymbol{h_1},\cdots,\boldsymbol{h_n}\rrbracket$, where $\boldsymbol{h_i}=(u_i\boldsymbol{1}_{\lceil d_i/2\rceil}, v_i\boldsymbol{1}_{\lfloor d_i/2\rfloor})$ $0<u_i,v_i\leq 1$. Let $\Omega\sim \mathcal{H}$. Then we repeat the process as in Section \ref{sim: uniform}. The simulation results are shown in Figures \ref{fig:T3_rk_error_nonUniform} and
\ref{fig:T4_rk_error_nonUniform}. 
As shown in figures, the results using our proposed weighted samples perform more efficiently than using the $p$ weighted samples.
\begin{figure}[H]
\centering
		\includegraphics[width=10cm]
		{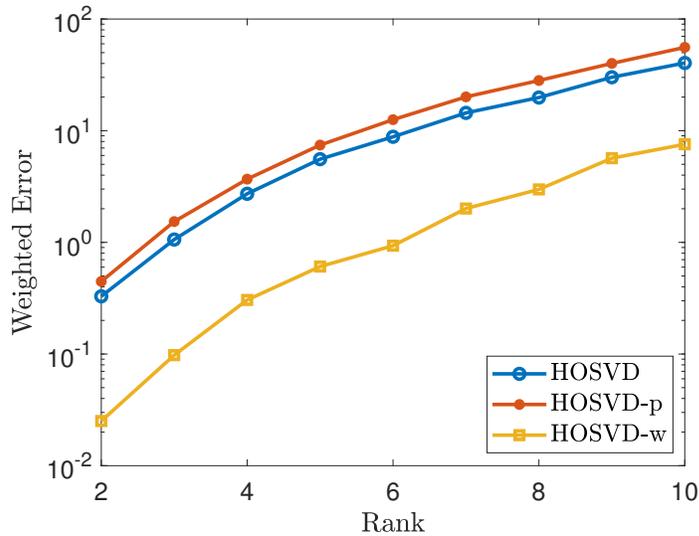}
		\label{FIG:T3_Weighted_error_N}
	  \caption{Tensor of size $100\times 100\times 100$ using the non-uniform sampling pattern: 
	  plots the errors of the form $\frac{\|\mathcal{W}^{(1/2)}\hadam(\widehat{\mathcal{T}}-\mathcal{T})\|_F}{\|\mathcal{W}^{(1/2)}\|_F}$. 
	  The lines labeled as HOSVD, HOSVD-p and HOSVD-w represent the results for solving \eqref{eqn:rgl_obj}, \eqref{eqn:std_obj} and \eqref{eqn:bia_obj}, respectively. }
    \label{fig:T3_rk_error_nonUniform}
	\end{figure}
	\begin{figure}[H]
	\centering
		\includegraphics[width=10cm]
		{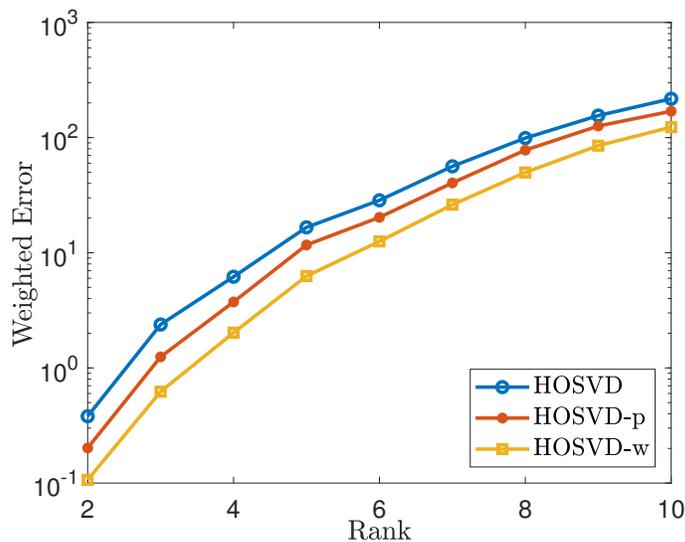}
	  \caption{Tensor of size $50\times 50\times 30\times 30$ using the non-uniform sampling pattern: 
	  plots the errors of the form $\frac{\|\mathcal{W}^{(1/2)}\hadam(\widehat{\mathcal{T}}-\mathcal{T})\|_F}{\|\mathcal{W}^{(1/2)}\|_F}$. 
	  The lines labeled as HOSVD, HOSVD-p and HOSVD-w represent the results for solving \eqref{eqn:rgl_obj}, \eqref{eqn:std_obj} and \eqref{eqn:bia_obj}, respectively. }
    \label{fig:T4_rk_error_nonUniform}
	\end{figure}


\begin{Remark}
	When we use the HOSVD procedure to solve \eqref{eqn:rgl_obj}, \eqref{eqn:std_obj}, and  \eqref{eqn:bia_obj}, we need (an estimate of) the Tucker rank as input. Instead of inputting the real rank of the true tensor, we could also use the rank that is estimated by considering the decay of the singular values for the unfolded matrices of the sampled tensor $\mathcal{Y}_{\Omega}$ as the input rank, which we call SV-rank. The simulation results for the non-uniform sampling pattern with SV-rank as input are reported in Figure \ref{fig:T3_rk_error_nonUniform_SV}. The simulation shows that the weighted HOSVD algorithm  performs more efficiently than using the $p$ weighted samples or the original samples. Comparing Figure \ref{fig:T3_rk_error_nonUniform_SV} with Figure \ref{fig:T3_rk_error_nonUniform}, we could observe that using the estimated rank as input for HOSVD procedure performs  even better than using the real rank as input. This observation motivates a way to find a ``good" rank as input for HOSVD procedure.
	\end{Remark}
\begin{Remark}
We only provide guarantees on the performance in the weighted Frobenius
norm, (as we report the weighted error $\frac{\|\mathcal{W}^{(1/2)}\hadam(\widehat{\mathcal{T}}-\mathcal{T})\|_F}{\|\mathcal{W}^{(1/2)}\|_F}$), our procedures exhibit good empirical performance even in the usual relative error  $\frac{\|\widehat{\mathcal{T}}-\mathcal{T}\|_F}{\|\mathcal{T}\|_F}$ when the Tucker rank of the tensor is relatively low. However, we observe that the advantages of weighted HOSVD scheme tend to be diminished in terms of relative error when the Tucker rank increases. This result is not surprising since the entries are treated unequally in scheme \eqref{eqn:bia_obj}. 
Therefore we leave the investigation on relative error and the tensor rank for future work.

	\begin{figure}[H]
\centering
		\includegraphics[width=10cm]{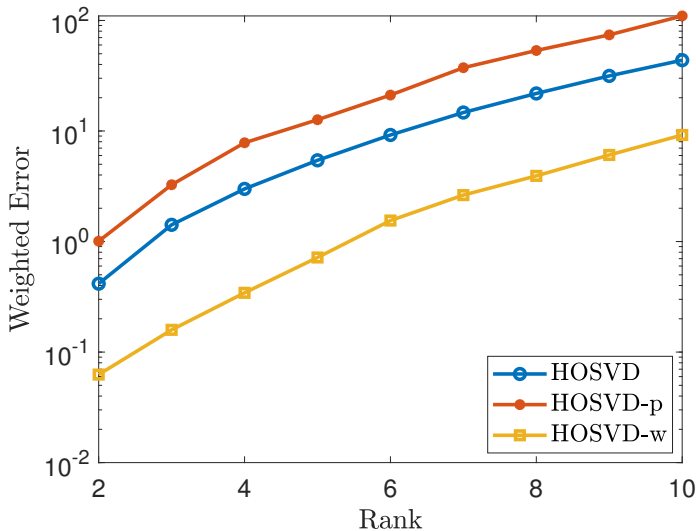}
	  \caption{Tensor of size $100\times 100\times 100$ using the non-uniform sampling pattern and with the SV-rank as the input rank: 
	  plots the errors of the form $\frac{\|\mathcal{W}^{(1/2)}\hadam(\widehat{\mathcal{T}}-\mathcal{T})\|_F}{\|\mathcal{W}^{(1/2)}\|_F}$. 
	  }
    \label{fig:T3_rk_error_nonUniform_SV}
    \end{figure}
\end{Remark}

\subsection{Test for Real Data}  \label{Test for real data}

In this section, we test our weighted HOSVD algorithm for tensor completion on three videos, see \cite{FHL2018}. The dataset is the tennis-serve data from an Olympic Sports Dataset \cite{niebles2010modeling}\footnote{One can download the dataset from 
\href{http://vision.stanford.edu/Datasets/OlympicSports/}{\bf\texttt{http://vision.stanford.edu/Datasets}}.
There are a lot of videos in the zip file and we only choose three of them: ``d2P\_zx\_JeoQ\_00120\_00515.seq” (video 1), ``gs3sPDfbeg4\_00082\_00229.seq”(video 2), and ``VADoc-AsyXk\_00061\_
0019.seq” (video 3).}. 
The three videos are color video. In our simulation, we use the same setup as the one in \cite{FHL2018}, and choose 30 frames evenly from each video. For each  frame, the size is scaled to $360\times 480\times 3$, so each video is transformed into a 4-D tensor data of size $360\times 480\times 3\times 30$. The first frame of each video after preprocessing is illustrated in Figure \ref{FIG: video}.

\begin{figure}[H]
    \centering
		\begin{subfigure}[b]{0.32\linewidth}
		\includegraphics[width=\textwidth]
		{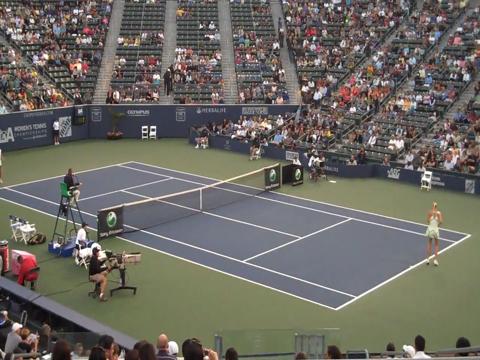}
		\caption{Video 1}
		\label{FIG:video1}
	\end{subfigure}
		\begin{subfigure}[b]{0.32\linewidth}
		\includegraphics[width=\textwidth]
		{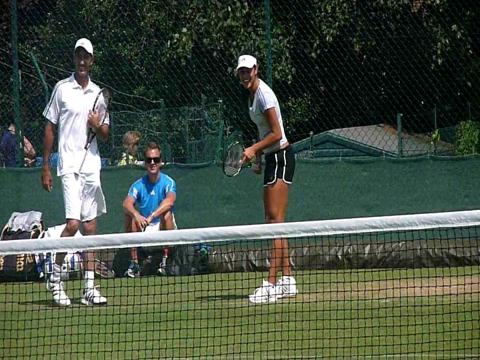}
		\caption{Video 2}
	\label{FIG:video2}
	\end{subfigure}
	\begin{subfigure}[b]{0.32\linewidth}
		\includegraphics[width=\textwidth]
		{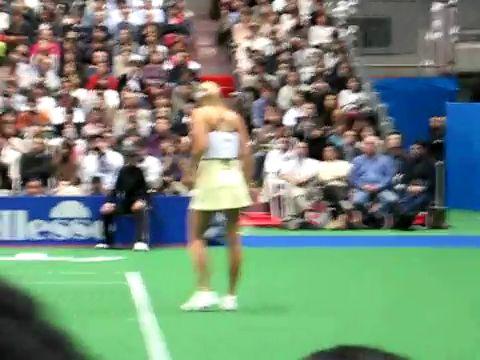}
		\caption{Video 3}
		\label{FIG:video3}
	\end{subfigure}
	  \caption{The first frame of videos \cite{FHL2018}. }
	  \label{FIG: video}
\end{figure}
We implement the experiments for different sampling rates of $10\%$, $30\%$, $50\%$, and $80\%$ to generate uniform and non-uniform sampling patterns $\Omega$. In our implementation, we  use the SV-rank of $\mathcal{T}_{\Omega}$ as the input rank. 
According to the generated sampling pattern, we find a weight tensor $\mathcal{W}$ and find estimates $\widehat{\mathcal{T}}_1$ and $\widehat{\mathcal{T}}_2$ by considering \eqref{eqn:rgl_obj} and \eqref{eqn:bia_obj} respectively, using the input Tucker rank $\boldsymbol{r}$. The entries on $\mathcal{T}_1$ and $\mathcal{T}_2$ are forced to be the observed data.  The Signal to Noise Ratio (SNR)
$$SNR(\widehat{\mathcal{T}}) = -20\log_{10}\left(\frac{\|\widehat{\mathcal{T}}-\mathcal{T}\|_F}{\|\mathcal{T}\|_F}\right)$$
are computed and the simulation results are reported in Tables \ref{tab:video} and \ref{tab2:video}. As shown in the tables, applying HOSVD process to \eqref{eqn:bia_obj} can give a better result than applying HOSVD process to \eqref{eqn:rgl_obj} directly regardless of the uniformity of the sampling pattern.
 
\begin{table}[!ht]
\renewcommand{\arraystretch}{1.2}
    \centering

    \begin{tabular}{|c|c|c|c|c|c|c|}\hline 
Video & SR & Input Rank & HOSVD\_w$+$TV & HOSVD & HOSVD\_w/HOSVD\_p & TVM \\ \hline

\multirow{4}{*}{\includegraphics[width=0.12\textwidth]{figures/1.jpg}}
& 10\% &$[7~17~3~5]$  & 13.29 (16.3s) & 1.27 (3.74s) &  10.15 (11.4s) &  13.04 (41.3s) \\
& 30\% &$[18~10~3~6]$ & 16.96 (14.0s) & 4.26 (4.01s) &  12.05 (7.23s) &  17.05 (29.7s) \\
& 50\% &$[26~4~3~11]$ & 19.60 (12.2s) & 8.21 (2.99s) &  14.59 (7.03s) &  19.68 (23.8s)\\
& 80\% &$[47~47~3~22]$& 24.90 (11.5s) & 17.29 (6.55s)&  19.75 (8.08s) &  25.01 (18.1s) \\
\hline

\multirow{4}{*}{\includegraphics[width=0.12\textwidth]{figures/2.jpg}} 

& 10\% &$[28~6~3~7]$   & 10.98 (13.1s)  &  1.19 (4.20s) & 7.88 (8.76s)  &  10.89 (42.2s) \\
& 30\% &$[34~18~3~15]$ & 14.44 (16.1s)  &  4.11 (3.80s) & 10.40 (7.51s)  &  14.50 (31.4s)\\
& 50\% &$[35~33~3~9]$  & 16.95 (15.3s) &  7.85 (5.86s) & 12.84 (7.64s) &  16.96 (26.6s)\\
& 80\% &$[56~50~3~21]$ & 22.21 (15.1s) & 16.51 (7.24s) & 18.64 (8.45s) &  22.19  (18.4s)\\
 \hline

\multirow{4}{*}{\includegraphics[width=0.12\textwidth]{figures/3.jpg}} 
& 10\% &$[12~9~3~10]$  & 12.34 (16.1s) & 1.22  (2.73s)&  8.46 (9.88s) &  12.23 (45.7s)\\
& 30\% &$[20~24~3~11]$ & 17.10 (15.3s) & 4.24  (3.17s)&  11.62 (7.62s) &  17.19 (35.3s) \\
& 50\% &$[25~32~3~14]$ & 20.44 (12.3s) & 8.20  (3.92s)&  14.54 (5.85s) &  20.49 (28.9s)\\
& 80\% &$[50~72~3~30]$ & 26.80 (12.4s) & 18.03 (8.40s)&  21.38 (8.93s)&  26.71 (20.9s)\\
 \hline
    \end{tabular}
        \caption{Signal to noise ratio (SNR) and elapsed time (in second) for HOSVD and HOSVD\_w on video data with uniform sampling pattern. The HOSVD\_w and HOSVD\_p behave very similar for uniform sampling hence we integrate the results into one column.}
    \label{tab:video}
\end{table}

\begin{table}[!ht]
\renewcommand{\arraystretch}{1.2}
    \centering
    \begin{tabular}{|c|c|c|c|c|c|}\hline
Video & SR &Input Rank& HOSVD & HOSVD\_w &HOSVD\_p  \\ \hline
\multirow{4}{*}{\includegraphics[width=0.12\textwidth]{figures/1.jpg}}
& 10\%&$\left[
6~13~3~3
\right]$   & 1.09  &  10.07  & 5.56\\
 &   30\%&$\left[
10~28~3~16
\right]$   & 3.74 & 11.81  & 7.53\\
 &   50\%& $\left[
21~41~3~14
\right]$  & 7.05 & 13.22 & 10.73\\
&80\%& $\left[44  ~57 ~    3 ~   26
\right]$& 15.76    & 19.60 & 17.39
\\
 \hline
\multirow{4}{*}{\includegraphics[width=0.12\textwidth]{figures/2.jpg}} 
& 10\%& $\left[
38~11~3~2
\right]$ & 1.13 &  8.04 & 4.33\\

 &   30\% &$\left[
26~19~3~16
\right]$ &3.79 & 10.13 & 6.80 \\
 &  50\% & $\left[
30~27~3~10
\right]$& 7.15 & 12.57 &10.14 \\
&80\%&$\left[53  ~ 50~     3~    23
\right]$& 14.81   & 18.55 &16.31\\

 \hline

\multirow{4}{*}{\includegraphics[width=0.12\textwidth]{figures/3.jpg}} 
 & 10\%&$\left[
16~11~3~2
\right]$  & 1.09 & 8.31 &4.73   \\
 & 30\% &$\left[
17~23~3~17
\right]$  & 3.76 & 11.05 &6.87 \\
 &  50\%  &$\left[
24~38~3~14
\right]$ & 7.18 & 13.78 &9.99 \\
&80\%& $\left[
47~69~3~22
\right]$& 15.88 &   20.82 &16.02\\

 \hline
    \end{tabular}
        \caption{Signal to noise ratio (SNR) for HOSVD and HOSVD\_w  on video data with non-uniform sampling pattern.}
    \label{tab2:video}
\end{table}

Finally, we test the proposed  weighted HOSVD algorithm on real candle video data named ``candle\_4\_A''\footnote{The dataset can be downloaded from the Dynamic Texture Toolbox in \href{http://www.vision.jhu.edu/code/}{http://www.vision.jhu.edu/code/})}. 
We have tested the relation between the relative errors and the sampling rates using $\boldsymbol{r}=(5,5,5)$ as the input rank for HOSVD algorithm. The relative errors are presented in Figure \ref{fig:RE_SR_video_candle4A}. The simulation results also show that the proposed weighted HOSVD algorithm can implement tensor completion efficiently.

\begin{figure}[H]
\centering
    \begin{subfigure}[b]{0.49\linewidth}
    \includegraphics[width=\textwidth]{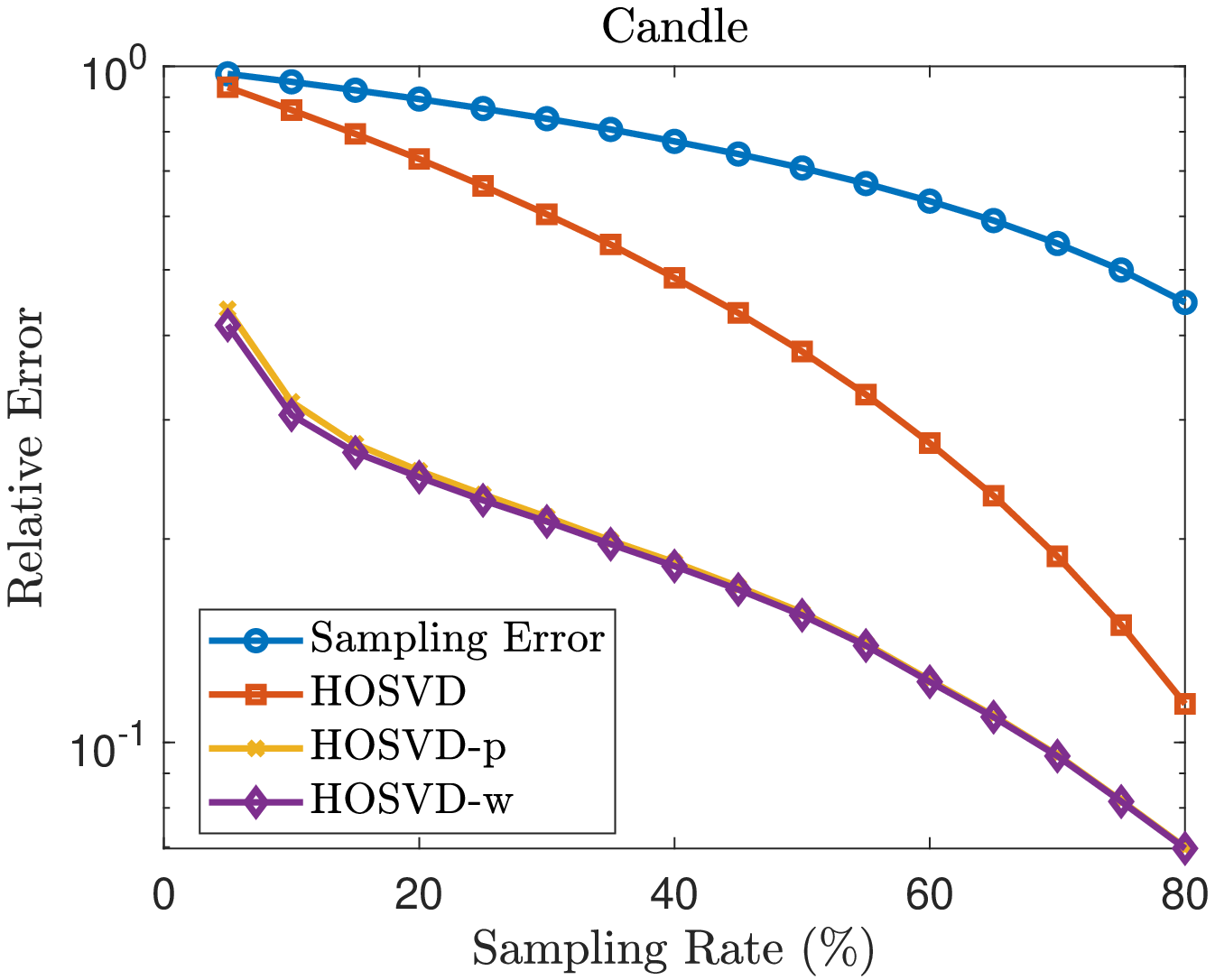}
    \end{subfigure}
    \begin{subfigure}[b]{0.49\linewidth}
     \includegraphics[width=\textwidth]{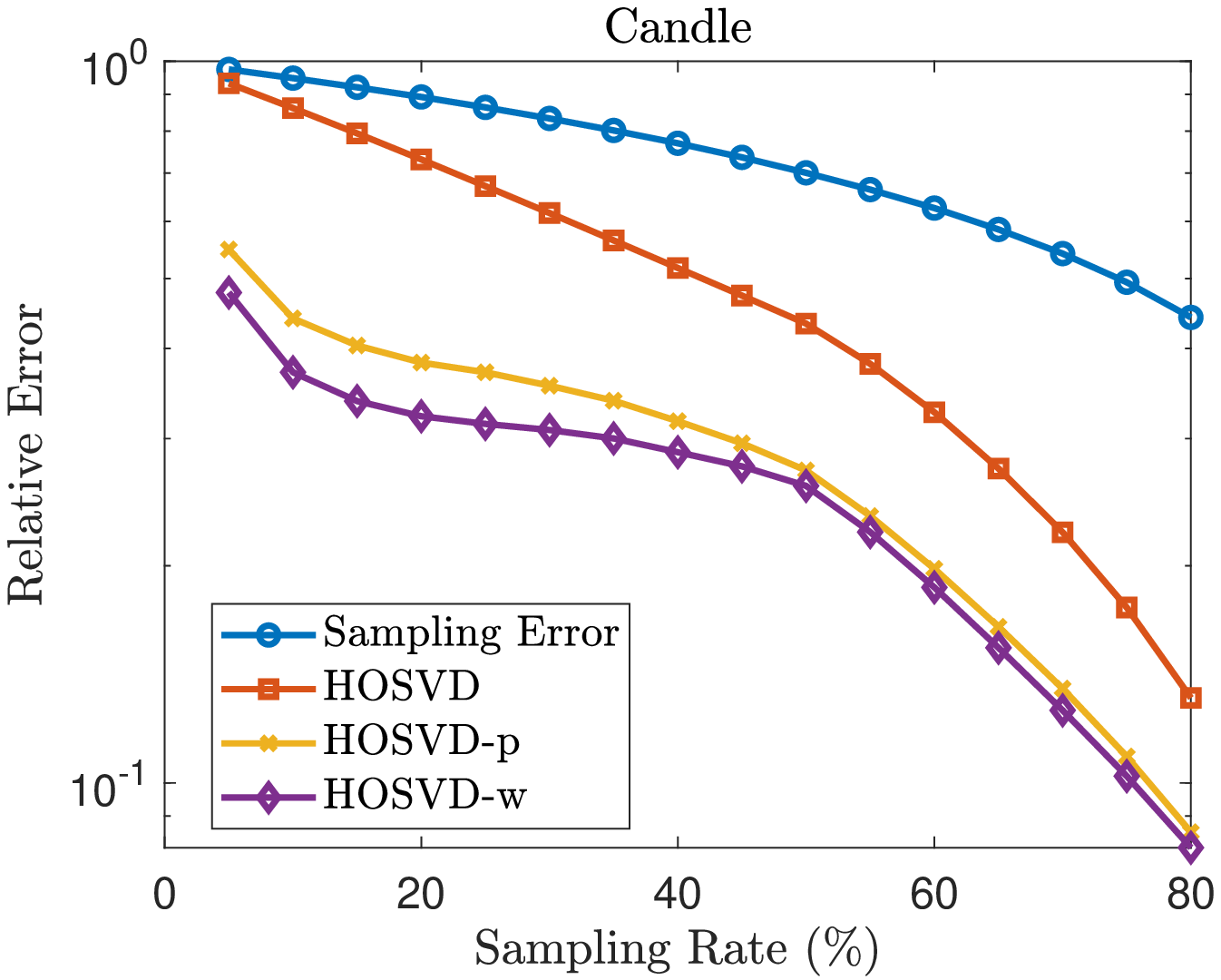}
    \end{subfigure}
    \caption{Relation between relative error and sampling rate for the dataset ``candle\_4\_A'' using $[5,5,5]$ as the input rank for HOSVD process. The left figure records the relative error for the uniform sampling pattern and the right figure for the non-uniform sampling pattern.  {The sampling error stands for the relative error between the original video and the video with masked entries estimated to be zeros, hence should approximately equal to $\sqrt{1-SR}$, where SR is the sampling rate}.} 
    \label{fig:RE_SR_video_candle4A}
\end{figure}
\subsection{The Application of  Weighted HOSVD on Total Variation Minimization}
As shown in the previous simulations, the weighted HOSVD decomposition can provide better results for tensor completion by comparing with HOSVD. There are a bunch of algorithms that are Sensitive to initialization. Additionally, real applications may have higher requirements for accuracy. Therefore, it is meaningful to combine our weighted HOSVD with other algorithms in order to further improve the performance. In this section, we would consider the application of weighted HOSVD decomposition on the total variation minimization algorithm. 
As a traditional approach, the total variation  minimization (TVM), 
is broadly applied in studies about image recovery and denoising. While the earliest research could trace back to 1992 \cite{rudin1992nonlinear}. The later studies combined TVM and other low rank approximation algorithms such as Nuclear Norm Minimization (see e.g., \cite{wu2017structure, madathil2018twist, yao2015accelerated}) and HOSVD (e.g.,  \cite{wang2017hyperspectral, ji2016tensor, li2017low}) in order to achieve better performance in image and video completion tasks.

Motivated by the matrix TV minimization, we proposed the tensor TV minimization which  is summarized in Algorithm \ref{alg:TCTV}.
In Algorithm \ref{alg:TCTV}, the Laplacian operator computes the divergence of all-dimension gradients for each entry of the tensor. The shrink operator simply moves the input towards 0 with distance $\lambda$, or formally defined as:
\[\text{shrink}(x,\lambda) = \mathbf{sign}(x)\cdot \max(|x|-\lambda, 0)\]

For the initialization of $\tens{X}^0$  in Algorithm \ref{alg:TCTV}, we assign $\tens{X}^{0}$ to be the output of the result from HOSVD-w. Applying the same experiment setting as in Section \ref{Test for real data}, we evaluate the performance of the cocktail approach as well as the regular HOSVD approach. We report the simulation results in Table \ref{tab:video} and we measure  the performances by considering the signal to noise ratio(SNR).
As shown in Table \ref{tab:video}, the total variation  minimization could be applied to further improve the result of \eqref{eqn:bia_obj}. Specifically, the TVM with $\mathbf{0}$ as initialization performs similar to TVM with HOSVD-w as initialization when the observed rate is high, but the HOSVD-w initialization could improve the performance of TVM when the observed rate is very low (e.g., $10\%$). Additionally, we compared the decay of relative error for using the weighted HOSVD output as initialization and the default initialization ($\tens{X}^0 = \mathbf{0}$). The iterative results are shown in Figure \ref{fig:convergence_compar}, and it shows that  using the result from weighted HOSVD as an initialization could notably reduce the iterations of TV-minimization for achieving the convergence threshold ($\|\tens{X}^k-\tens{X}^{k-1}\|_F<10^{-4}$).

\vspace{2mm}
\begin{algorithm*}[H]
\SetKwInOut{Input}{Input}\SetKwInOut{Output}{Output}
\Input{Noised tensor $\tens{T}\in\mathbb{R}^{d_1\times \dots \times d_n}$; Sampling pattern $\Omega\in\{0,1\}^{d_1\times \dots \times d_n}$;
stepsize $h_k$, threshold $\lambda$;
$\tens{X}^0\in\mathbb{R}^{d_1\times\cdots\times d_n}$.}

{Set $\tens{X}^0 = \mathcal{X}^0+(\tens{T}_\Omega-\tens{X}^0_{\Omega})$.}

\For{$k = 0:K$}{
    \For{$i = 1:n$}{
        $\nabla_i(\tens{X}^k_{\alpha_1,...,\alpha_n})=\tens{X}^k_{\alpha_1,...,\alpha_i+1,...,\alpha_n}-\tens{X}^k_{\alpha_1,...\alpha_i,...,\alpha_n},  (\alpha_i=1, 2,...,d_i-1)$
        ($\nabla_i(\cdot) = 0 \text{ when } \alpha_i = d_i$)
        
        $\Delta_i(\tens{X}^k_{\alpha_1,...,\alpha_n})=\tens{X}^k_{\alpha_1,...,\alpha_i-1,...,\alpha_n}+\tens{X}^k_{\alpha_1,...,\alpha_i+1,...,\alpha_n}-2\tens{X}^k_{\alpha_1,...\alpha_i,...,\alpha_n}, (\alpha_i=2,3,...,d_i-1)$
        ($\Delta_i(\cdot) = 0 \text{ when } \alpha_i = 1$ or $d_i$)
    }
    $\Delta(\tens{X}^k_{\alpha_1,...,\alpha_n}) = \sum_i\Delta_i(\tens{X}^k_{\alpha_1,...,\alpha_n})$\\
    $\tens{X}^{k+1}_{\alpha_1,...,\alpha_n} = \tens{X}^k_{\alpha_1,...,\alpha_n}+h_k\cdot\text{shrink}(\frac{\Delta(\tens{X}^k_{\alpha_1,...,\alpha_n})}{\sqrt{\sum_i\nabla_i^2(\tens{X}^k_{\alpha_1,...,\alpha_n})}},\lambda)$\\
    $\tens{X}_\Omega^{k+1} = \tens{T}_\Omega$
    }
\Output{$\tens{X}^K$}
\caption{TV  Minimization for Tensor}
\label{alg:TCTV}
\end{algorithm*}
\vspace{2mm}

\begin{figure}[H]
		\begin{subfigure}[b]{0.49\linewidth}
		\includegraphics[width=\textwidth]
		{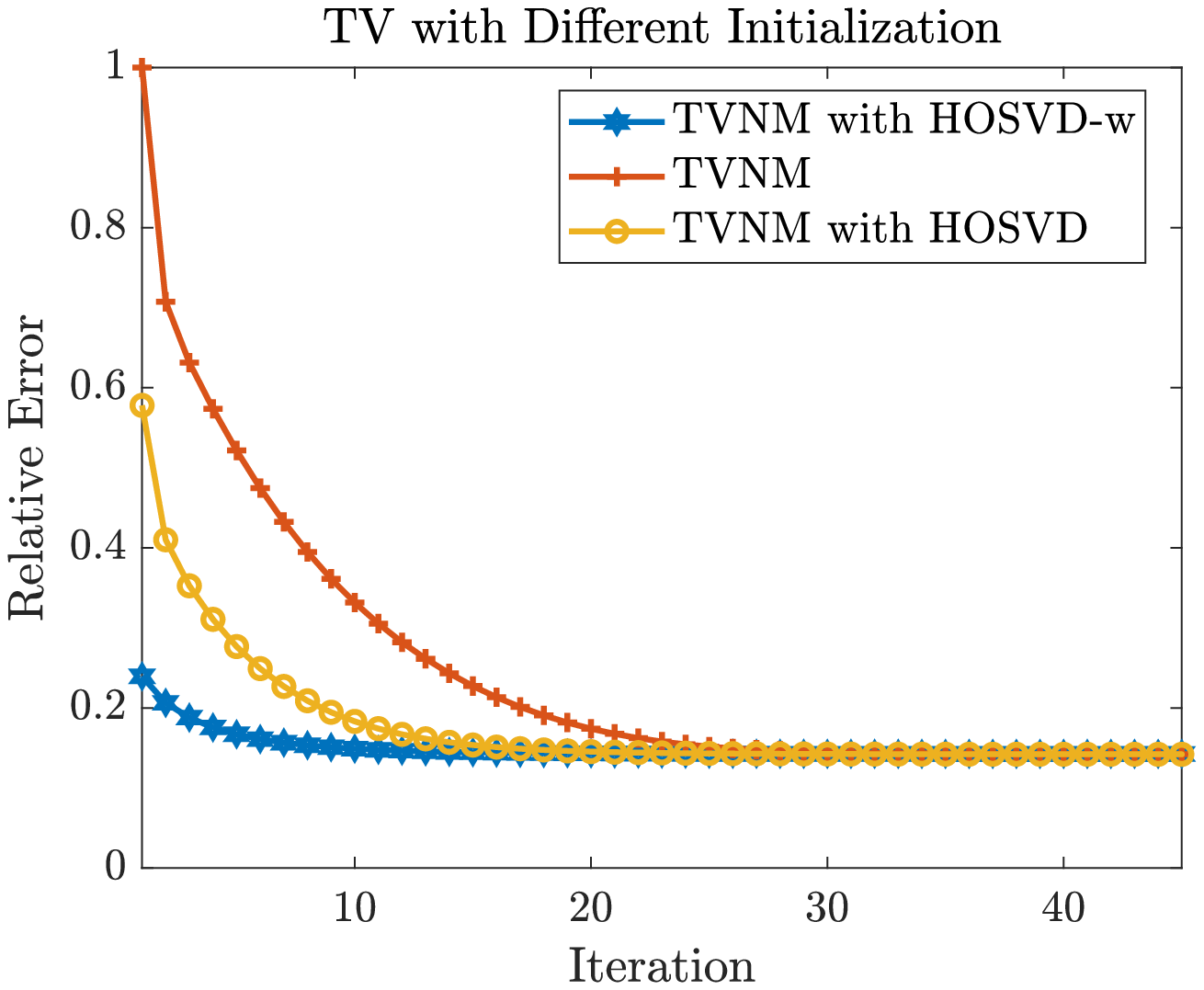}
		\caption{}
		\label{FIG:cc_iter}
	\end{subfigure}
		\begin{subfigure}[b]{0.49\linewidth}
		\includegraphics[width=\textwidth]
		{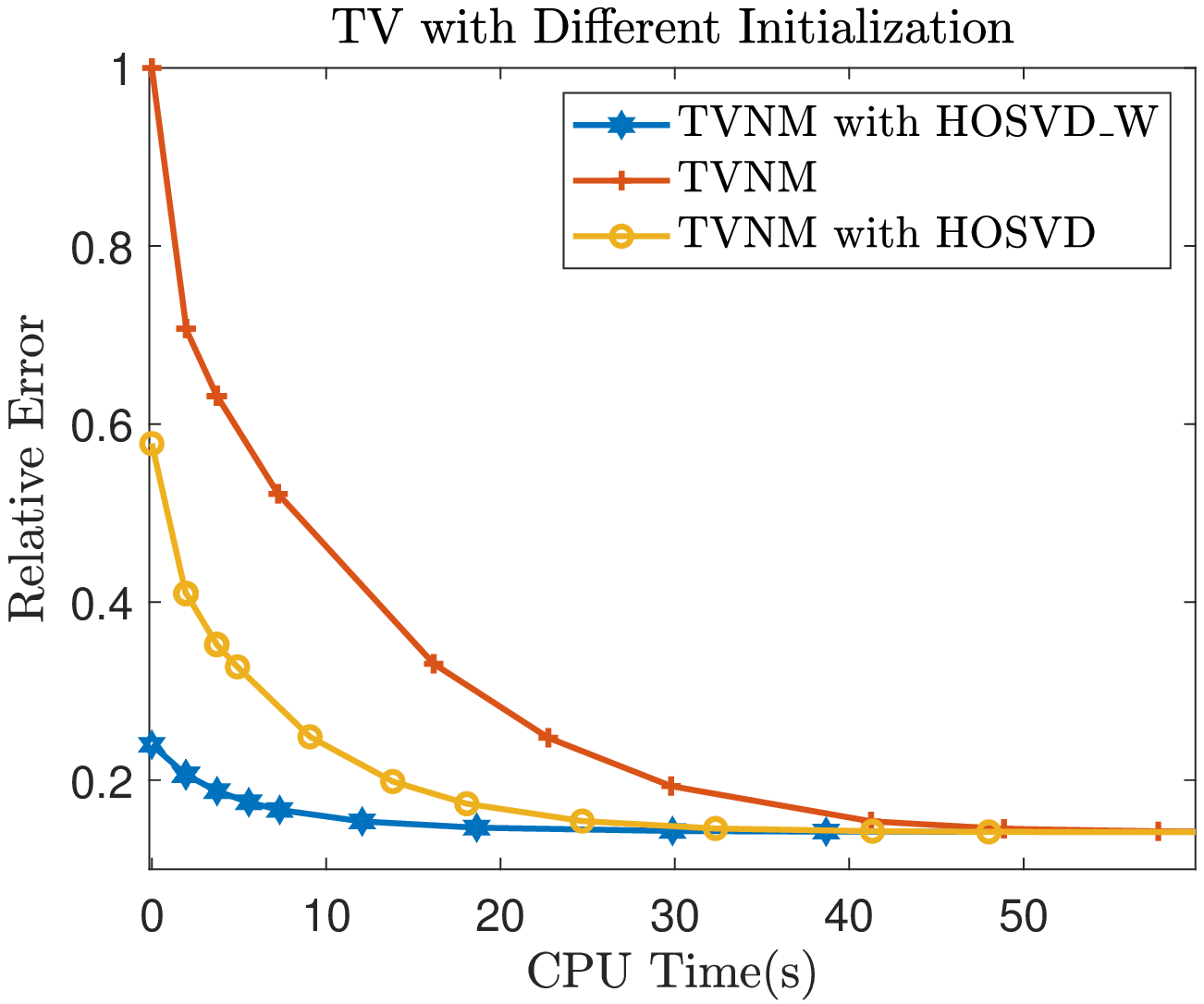}
		\caption{}
		\label{FIG:cc_cpu}
	\end{subfigure}
	  \caption{Convergence comparison between total variation
minimization (TVM) with HOSVD-w, $\mathbf{0}$, and HOSVD  as initialization on video 1 with SR = 50\%:
	  (\textbf{a})   the relative error $\frac{\|\widehat{\mathcal{T}}-\mathcal{T}\|_F}{\|\tens{T}\|_F}$ vs. number of iterations. (\textbf{b})   the relative error v.s. total computational CPU time(initialization + completion).}
    \label{fig:convergence_compar}
    \end{figure}
    
\section{Conclusions}
In this paper, we  propose a simple but efficient algorithm named the weighted HOSVD algorithm for recovering an underlying low-rank tensor from noisy observations. For this algorithm, we  provide  upper and lower error bounds that measure the difference between the estimates and the true underlying low-rank tensor. The efficiency of our proposed weighted HOSVD algorithm is also shown by numerical simulations. Additionally, the result of our weighted HOSVD algorithm can be used as an initialization for the total variation minimization algorithm,
which shows that using our method as an initialization for the total variation minimization algorithm can increasingly reduce the iterative steps leading to improved overall performance in reconstruction (see our conference paper \cite{CHN2020}).  It would be interesting for future work to combine the weighted HOSVD algorithm with other algorithms to achieve more accurate results for tensor completion in many settings.

\section*{Acknowledgements}
 The  authors are supported by NSF DMS $\#2011140$ and NSF BIGDATA $\#1740325$. The authors take pleasure in thanking Hanqin Cai,  Keaton Hamm, Armenak Petrosyan, Bin Sun, and  Tao Wang for comments and suggestions on  the manuscript.

\appendix

\section{ Proof for Theorem 1}\label{section:proofFgub}
In this appendix, we provide the proof for Theorem \ref{thm:gub tensor}.
\begin{proof}[Proof of Theorem \ref{thm:gub tensor}]
Let $\mathcal{Y}_\Omega=\mathcal{T}_{\Omega}+\mathcal{Z}_{\Omega}$.
\begin{eqnarray*}
&&\left\|\mathcal{W}^{(1/2)}\hadam(\mathcal{T}-\widehat{\mathcal{T}})\right\|_F\\
&=&\left\|\mathcal{W}^{(1/2)}\hadam\mathcal{T}-\mathcal{W}^{(-1/2)}\hadam\mathcal{Y}_\Omega+\mathcal{W}^{(-1/2)}\hadam\mathcal{Y}_\Omega-\mathcal{W}^{(1/2)}\hadam\widehat{\mathcal{T}}\right\|_F\\
&\leq&\left\|\mathcal{W}^{(1/2)}\hadam\mathcal{T}-\mathcal{W}^{(-1/2)}\hadam\mathcal{Y}_\Omega\right\|_F+\left\|\mathcal{W}^{(-1/2)}\hadam\mathcal{Y}_\Omega-\mathcal{W}^{(1/2)}\hadam\widehat{\mathcal{T}}\right\|_F\\
&\leq&2\left\|\mathcal{W}^{(1/2)}\hadam\mathcal{T}-\mathcal{W}^{(-1/2)}\hadam\mathcal{Y}_\Omega\right\|_F\\
&=&2\left\|\mathcal{W}^{(1/2)}\hadam\mathcal{T}-\mathcal{W}^{(-1/2)}\hadam(\mathcal{T}_\Omega+\mathcal{Z}_{\Omega})\right\|_F\\
&\leq&2\left\|\mathcal{W}^{(1/2)}\hadam\mathcal{T}-\mathcal{W}^{(-1/2)}\hadam \boldsymbol{1}_{\Omega}\hadam\mathcal{T}\right\|_F+2\left\|\mathcal{W}^{(-1/2)}\hadam\mathcal{Z}_{\Omega}\right\|_F\\
&\leq&2\left\|\mathcal{T}\hadam(\mathcal{W}^{(1/2)}-\mathcal{W}^{(-1/2)}\hadam \boldsymbol{1}_{\Omega})\right\|_F+2\left\|\mathcal{W}^{(-1/2)}\hadam\mathcal{Z}_{\Omega}\right\|_F\\
&\leq&2\left\|\mathcal{T}\right\|_{\infty}\left\|\mathcal{W}^{(1/2)}-\mathcal{W}^{(-1/2)}\hadam \boldsymbol{1}_{\Omega}\right\|_F+2\left\|\mathcal{W}^{(-1/2)}\hadam\mathcal{Z}_{\Omega}\right\|_F.
\end{eqnarray*}

Thus, we have that 
\begin{equation}\label{eqn:rkr}
    \left\|\mathcal{W}^{(1/2)}\hadam(\mathcal{T}-\widehat{\mathcal{T}})\right\|_F\leq 2\left\|\mathcal{T}\right\|_{\infty}\left\|\mathcal{W}^{(1/2)}-\mathcal{W}^{(-1/2)}\hadam \boldsymbol{1}_{\Omega}\right\|_F+2\left\|\mathcal{W}^{(-1/2)}\hadam\mathcal{Z}_{\Omega}\right\|_F.
\end{equation}

Next, let's estimate $\left\|\mathcal{W}^{(-1/2)}\hadam\mathcal{Z}_{\Omega}\right\|_F$.
Notice that
\[\left\|\mathcal{W}^{-(1/2)}\hadam\mathcal{Z}_{\Omega}\right\|_F^2=\sum_{(i_1,\cdots,i_n)\in\Omega}\frac{\mathcal{Z}_{i_1\cdots i_n}^2}{\mathcal{W}_{i_1\cdots i_n}}
\]
\begin{eqnarray*}
\mathbb{P}\left\{\left\|\mathcal{W}^{(-1/2)}\hadam\mathcal{Z}_{\Omega}\right\|_F\geq t\right\}&=&\mathbb{P}\left\{e^{s\left\|\mathcal{W}^{(-1/2)}\hadam\mathcal{Z}_{\Omega}\right\|_F^2}\geq e^{st^2}\right\}\\
&\leq&e^{-st^2}\mathbb{E}\left(\exp\left({s\left\|\mathcal{W}^{(-1/2)}\hadam\mathcal{Z}_{\Omega}\right\|_F^2}\right)\right)\\
&\leq&e^{-st^2}\prod\limits_{(i_1,\cdots,i_n)\in\Omega}\mathbb{E}\left(\exp\left({\frac{s\mathcal{Z}_{i_1\cdots i_n}^2}{\mathcal{W}_{i_1\cdots i_n}}}\right)\right)\\
&=&e^{-st^2}\prod\limits_{(i_1,\cdots,i_n)\in\Omega}\left(\frac{1}{\sqrt{1-2\sigma^2s/\mathcal{W}_{i_1\cdots i_n}}}\right)
\end{eqnarray*}

Recall that  $\mu^2=\max_{(i_1,\cdots,i_n)\in\Omega}\frac{1}{\mathcal{W}_{i_1,\cdots,i_n}}$. By choosing $s=\frac{1}{4\sigma^2\mu^2}$, we have that
\begin{eqnarray*}
\mathbb{P}\left\{\left\|\mathcal{W}^{-(1/2)}\circ\mathcal{Z}_{\Omega}\right\|_F\geq t\right\}\leq \exp\left({-\frac{t^2}{4\sigma^2\mu^2}}\right)2^{|\Omega|/2}.
\end{eqnarray*}

We conclude that with probability at least $1-2^{-|\Omega|/2}$, 
\[\left\|\mathcal{W}^{(-1/2)}\circ \mathcal{Z}_{\Omega}\right\|_F\leq 2\sigma \mu\sqrt{|\Omega|\log(2)}.
\]

Plugging this into \eqref{eqn:rkr} proves the theorem.
\end{proof}
\section{Proof of Theorems 2 and 3}\label{section:proofFHOSVD}
In this appendix, we provide the proofs for the results related with the weighted HOSVD algorithm. The general upper bound for weighted HOSVD in  Theorem \ref{thm: gubwHOSVD} is restated  in Appendix \ref{appendixB1} and its proof is also presented there. If the sampling pattern $\Omega$ is generated according to the weight tensor $\mathcal{W}$, the related results in Theorem \ref{thm:bfsomega} are illustrated in Appendix \ref{appendixB2}.   
\subsection{General Upper Bound for Weighted HOSVD Algorithm}\label{appendixB1}
\begin{theorem}\label{thm:ub low-rank tensor}
Let $\mathcal{W}=\boldsymbol{w_1}\out\cdots\out \boldsymbol{w_n}\in\mathbb{R}^{d_1\times\cdots \times d_n}$ have strictly positive entries, and fix $\Omega\subseteq[d_1]\times\cdots\times[d_n]$. Suppose that $\mathcal{T}\in\mathbb{R}^{d_1\times \cdots\times d_n}$ has Tucker rank $\textbf{r}=[r_1~\cdots~ r_n]$. Suppose that $\mathcal{Z}_{i_1\cdots i_n}\sim\mathcal{N}(0,\sigma^2)$ and let 
\[\widehat{\mathcal{T}}=\mathcal{W}^{(-1/2)}\hadam((\mathcal{W}^{(-1/2)}\hadam\mathcal{Y}_\Omega)\times_1 \widehat{U}_1\widehat{U}_1^T\times_2 \cdots\times_n \widehat{U}_n\widehat{U}_n^T)
\]
where $\widehat{U}_1, \cdots, \widehat{U}_n$ are obtained by  HOSVD approximation process, where $\mathcal{Y}_{\Omega}=\boldsymbol{1}_{\Omega}\hadam(\mathcal{T}+\mathcal{Z})$.
Then with probability at least
$1-\sum_{i=1}^{n}\frac{1}{d_i+\prod_{j\neq i}d_j}$
over the choice of $\mathcal{Z}$, 
\begin{align*}\label{uppb}
   & \left\|\mathcal{W}^{(1/2)}\hadam(\mathcal{T}-\widehat{\mathcal{T}})\right\|_F\\
   \leq& \left(\sum_{k=1}^n 6\sqrt{r_k\log(d_k+\prod\limits_{j\neq k}d_j)}\mu_k\right)\sigma
   +\left(\sum_{k=1}^{n}3r_k\left\|(\mathcal{W}^{(-1/2)}\hadam \boldsymbol{1}_{\Omega}-\mathcal{W}^{(1/2)})_{(k)}\right\|_2 \right)\left\|\mathcal{T}\right\|_\infty. 
\end{align*}
where \[\mu_k^2=\max\left\{\max_{i_k}\left(\sum_{i_1,\cdots,i_{k-1},i_{k+1},\cdots,i_n}\frac{1_{(i_1,i_2,\cdots,i_n)\in\Omega}}{\mathcal{W}_{i_1i_2\cdots i_{n}}}\right),\max_{i_1,\cdots,i_{k-1},i_{k+1},\cdots,i_n}\left(\sum_{i_k}\frac{1_{(i_1,i_2,\cdots,i_n)\in\Omega}}{\mathcal{W}_{i_1i_2\cdots i_{n}}}\right) \right\}.\]
\end{theorem}
\begin{proof}
Recall that $\mathcal{T}_{\Omega}=\boldsymbol{1}_{\Omega}\hadam\mathcal{T}$ and $\mathcal{Z}_{\Omega}=\boldsymbol{1}_{\Omega}\hadam\mathcal{Z}$.
First  we have the following estimations.
\begin{eqnarray*}
&&\left\|\mathcal{W}^{(1/2)}\hadam\left(\widehat{\mathcal T}-\mathcal{T}\right)\right\|_F\\
&=&\left\|\left(\mathcal{W}^{(-1/2)}\hadam\mathcal{Y}_\Omega\right)\times_1\widehat{U}_1\widehat{U}_1^T\times_2\cdots\times_n\widehat{U}_n\widehat{U}_n^T-\left(\mathcal{W}^{(1/2)}\hadam\mathcal{T}\right)\times_1U_1U_1^T\times_2\cdots \times_nU_nU_n^T\right\|_F\\
&\leq&\left\|\left((\mathcal{W}^{(-1/2)}\hadam\mathcal{Y}_\Omega)\times_1 \widehat{U}_1\widehat{U}_1^T-(\mathcal{W}^{(1/2)}\hadam \mathcal{T})\times_1U_1U_1^T\right)\times_2 \widehat{U}_2\widehat{U}_2^T\times_3\cdots\times_n\widehat{U}_n\widehat{U}_n^T\right\|_F\\
&&+\left\|\left(\mathcal{W}^{(1/2)}\hadam\mathcal{T}\right)\left(\times_2U_2U_2^T\times_3\cdots\times_nU_nU_n^T-\times_2\widehat{U}_2\widehat{U}_2^T\times_3\cdots\times_n\widehat{U}_n\widehat{U}_n^T\right)\right\|_F\\
&\leq&\sqrt{2r_1}\left\|\widehat{U}_1\widehat{U}_1^T(\mathcal{W}^{(-1/2)}\hadam\mathcal{Y}_\Omega)_{(1)}-U_1U_1^T(\mathcal{W}^{(1/2)}\hadam \mathcal{T})_{(1)}\right\|_2+\sum_{k=2}^{n}\left\|(\mathcal{W}^{(1/2)}\hadam\mathcal{T})\right.\\
&&\left.\times_2 \widehat{U}_2\widehat{U}_2^T\times_3\cdots\times_{k-1}\widehat{U}_{k-1}\widehat{U}_{k-1}^T\times_k(U_kU_k^T-\widehat{U}_k\widehat{U}_k^T)\times_{k+1}U_{k+1}U_{k+1}^T\times_{k+2}\cdots\times_n U_nU_n^T\right\|_F\\
&\leq&\sqrt{2r_1}\left\|\widehat{U}_1\widehat{U}_1^T(\mathcal{W}^{(-1/2)}\hadam\mathcal{Y}_\Omega)_{(1)}-(\mathcal{W}^{(1/2)}\hadam \mathcal{T})_{(1)}\right\|_2+\sum_{k=2}^n\sqrt{r_k}\left\|(U_kU_k^T-\widehat{U}_k\widehat{U}_k^T)(\mathcal{W}^{(1/2)}\hadam\mathcal{T})_{(k)}\right\|_2\\
&\leq&\sqrt{2r_1}\left(\left\|\widehat{U}_1\widehat{U}_1^T(\mathcal{W}^{(-1/2)}\hadam\mathcal{Y}_\Omega)_{(1)}-(\mathcal{W}^{(-1/2)}\hadam\mathcal{Y}_\Omega)_{(1)}\right\|_2+\left\|(\mathcal{W}^{(-1/2)}\hadam\mathcal{Y}_\Omega)_{(1)}-(\mathcal{W}^{(1/2)}\hadam \mathcal{T})_{(1)}\right\|_2\right)\\
&&+\sum_{k=2}^n\sqrt{r_k}\left\|(U_kU_k^T-\widehat{U}_k\widehat{U}_k^T)(\mathcal{W}^{(1/2)}\hadam\mathcal{T})_{(k)}\right\|_2\\
&\leq&2\sqrt{2r_1}\left\|(\mathcal{W}^{(-1/2)}\hadam\mathcal{Y}_\Omega)_{(1)}-(\mathcal{W}^{(1/2)}\hadam \mathcal{T})_{(1)}\right\|_2+\sum_{k=2}^n\sqrt{r_k}\left\|(U_kU_k^T-\widehat{U}_k\widehat{U}_k^T)(\mathcal{W}^{(1/2)}\hadam\mathcal{T})_{(k)}\right\|_2.
\end{eqnarray*}
Notice that
\begin{eqnarray*}
&&\left\|\left(U_kU_k^T-\widehat{U}_k\widehat{U}_k^T\right)(\mathcal{W}^{(1/2)}\hadam\mathcal{T})_{(k)}\right\|_2\\
&=&\left\|(\mathcal{W}^{(1/2)}\hadam\mathcal{T})_{(k)}-\widehat{U}_k\widehat{U}_k^T(\mathcal{W}^{(1/2)}\hadam\mathcal{T})_{(k)}\right\|_2\\
&\leq&\left\|(\mathcal{W}^{(1/2)}\hadam\mathcal{T})_{(k)}-(\mathcal{W}^{(-1/2)}\hadam\mathcal{Y}_\Omega)_{(k)}\right\|_2+\left\|\widehat{U}_k\widehat{U}_k^T(\mathcal{W}^{(1/2)}\hadam\mathcal{T}-\mathcal{W}^{(-1/2)}\hadam\mathcal{Y}_\Omega)_{(k)}\right\|_2+\\
&&\left\|(\mathcal{W}^{(-1/2)}\hadam\mathcal{Y}_\Omega)_{(k)}-\widehat{U}_k\widehat{U}_k^T(\mathcal{W}^{(-1/2)}\hadam\mathcal{Y}_{\Omega})_{(k)}\right\|_2\\
&\leq&3\left\|(\mathcal{W}^{(1/2)}\hadam\mathcal{T})_{(k)}-(\mathcal{W}^{(-1/2)}\hadam\mathcal{Y}_\Omega)_{(k)}\right\|_2.
\end{eqnarray*}
Therefore, we have
\begin{eqnarray}\label{eqn:ube}
&&\left\|\mathcal{W}^{(1/2)}\hadam(\widehat{\mathcal{T}}-\mathcal{T})\right\|_F\leq\sum_{k=1}^n3\sqrt{r_k}\left\|(\mathcal{W}^{(1/2)}\hadam\mathcal{T})_{(k)}-(\mathcal{W}^{(-1/2)}\hadam\mathcal{Y}_\Omega)_{(k)}\right\|_2.
\end{eqnarray}

Next, to estimate  $\left\|(\mathcal{W}^{(-1/2)}\hadam\mathcal{Y}_\Omega-\mathcal{W}^{(1/2)}\hadam \mathcal{T})_{(k)}\right\|_2$ for $k=1,\cdots,n$.

Let us  consider the case when $k=1$. Other cases can be derived similarly. Using the fact that $\mathcal{T}_{(1)}$ has rank $r_1$ and $\left\|\mathcal{T}_{(1)}\right\|_{\max}\leq\sqrt{r_1}\left\|\mathcal{T}_{(1)}\right\|_\infty=\sqrt{r_1}\left\|\mathcal{T}\right\|_\infty$, we conclude that
\begin{eqnarray*}
&&\left\|(\mathcal{W}^{(-1/2)}\hadam\mathcal{Y}_\Omega-\mathcal{W}^{(1/2)}\hadam \mathcal{T})_{(1)}\right\|_2\\
&=&\left\|(\mathcal{W}^{(-1/2)}\hadam \mathcal{T}_{\Omega}-\mathcal{W}^{(1/2)}\hadam \mathcal{T})_{(1)}+(\mathcal{W}^{(-1/2)} \hadam\mathcal{Z}_{\Omega})_{(1)}\right\|_2\\
&\leq&\left\|(\mathcal{W}^{(-1/2)}\hadam\mathcal{\mathcal{T}}_\Omega-\mathcal{W}^{(1/2)}\hadam \mathcal{T})_{(1)}\right\|_2+\left\|(\mathcal{W}^{(-1/2)}\hadam\mathcal{Z}_{\Omega})_{(1)}\right\|_2\\
&=&\left\|(\mathcal{W}^{(-1/2)}\hadam \boldsymbol{1}_{\Omega}-\mathcal{W}^{(1/2)})_{(1)}\hadam\mathcal{T}_{(1)}\right\|_2+\left\|(\mathcal{W}^{(-1/2)}\hadam\mathcal{Z}_{\Omega})_{(1)}\right\|_2\\
&\leq&\left\|\mathcal{T}_{(1)}\right\|_{\max}\left\|(\mathcal{W}^{(-1/2)}\hadam \boldsymbol{1}_{\Omega}-\mathcal{W}^{(1/2)})_{(1)}\right\|_2+\left\|(\mathcal{W}^{(-1/2)}\hadam\mathcal{Z}_{\Omega})_{(1)}\right\|_2\\
&\leq&\sqrt{r_1}\left\|\mathcal{T}\right\|_{\infty}\left\|(\mathcal{W}^{(-1/2)}\hadam \boldsymbol{1}_{\Omega}-\mathcal{W}^{(1/2)})_{(1)}\right\|_2+\left\|(\mathcal{W}^{(-1/2)}\hadam\mathcal{Z}_{\Omega})_{(1)}\right\|_2.
\end{eqnarray*}

To bound $\left\|(\mathcal{W}^{(-1/2)}\hadam\mathcal{Z}_{\Omega})_{(1)}\right\|_2$, we consider
\[(\mathcal{W}^{(-1/2)}\hadam\mathcal{Z}_{\Omega})_{(1)}=\sum_{i_1,\cdots,i_n}\frac{1_{(i_1,\cdots,i_n)\in\Omega}\mathcal{Z}_{i_1\cdots i_n}}{\sqrt{\mathcal{W}_{i_1\cdots i_n}}}\boldsymbol{e_{i_1}}(\boldsymbol{e_{i_2}}\kron\cdots\kron \boldsymbol{e_{i_n}})^T,
\]
where $\boldsymbol{e_{i_k}}$ is the $i_k$-th standard basis vector of $\mathbb{R}^{d_k}$.

Please note that
\begin{eqnarray*}
&&\sum_{i_1,\cdots,i_n}\frac{1_{(i_1,\cdots,i_n)\in\Omega}}{\mathcal{W}_{i_1\cdots i_n}}\boldsymbol{e_{i_1}}(\boldsymbol{e_{i_2}}\kron \cdots\kron \boldsymbol{e_{i_n}})^T(\boldsymbol{e_{i_2}}\kron  \cdots\kron \boldsymbol{e_{i_n}})\boldsymbol{\boldsymbol{e_{i_1}}}^T\\
&=&\sum_{i_1,\cdots,i_n}\frac{1_{(i_1,\cdots,i_n)\in\Omega}}{\mathcal{W}_{i_1\cdots i_n}}\boldsymbol{\boldsymbol{e_{i_1}}}\boldsymbol{\boldsymbol{e_{i_1}}}^T.
\end{eqnarray*}

Therefore,
\begin{eqnarray*}
&&\left\|\sum_{i_1,\cdots,i_n}\frac{1_{(i_1,\cdots,i_n)\in\Omega}}{\mathcal{W}_{i_1\cdots i_n}}\boldsymbol{\boldsymbol{e_{i_1}}}(\boldsymbol{e_{i_2}}\kron\cdots\kron \boldsymbol{e_{i_n}})^T(\boldsymbol{e_{i_2}}\kron\cdots\kron \boldsymbol{e_{i_n}})\boldsymbol{\boldsymbol{e_{i_1}}}^T\right\|_2\\
&=&\max_{i_1}\sum_{i_2,\cdots,i_n}\frac{1_{(i_1,i_2,\cdots,i_n)\in\Omega}}{\mathcal{W}_{i_1i_2\cdots i_n}}\leq\mu_1^2.
\end{eqnarray*}

Similarly,
\begin{eqnarray*}
&&\left\|\sum_{i_1,\cdots,i_n}\frac{1_{(i_1,\cdots,i_n)\in\Omega}}{\mathcal{W}_{i_1\cdots i_n}}(\boldsymbol{e_{i_2}}\kron \cdots\kron \boldsymbol{e_{i_n}})\boldsymbol{\boldsymbol{e_{i_1}}}^T\boldsymbol{\boldsymbol{e_{i_1}}}(\boldsymbol{e_{i_2}}\kron \cdots\kron \boldsymbol{e_{i_n}})^T\right\|_2\\
&=&\max_{i_2,\cdots,i_n}\sum_{i_1}\frac{1_{(i_1,i_2,\cdots,i_n)\in\Omega}}{\mathcal{W}_{i_1i_2\cdots i_n}}\leq\mu_1^2.
\end{eqnarray*}

By  (\cite{tropp2012user} Theorem 1.5), for any $t>0$,
\begin{eqnarray*}\mathbb{P}\left\{ \left\|(\mathcal{W}^{(-1/2)}\hadam\mathcal{Z}_{\Omega})_{(1)}\right\|\geq t \right\}\leq\left(d_1+\prod\limits_{j\neq 1}d_j \right)\exp\left(-\frac{t^2}{2\sigma^2\mu_1^2}\right).
\end{eqnarray*}

We conclude that with probability at least $1-\frac{1}{d_1+\prod_{j\neq 1}d_j}$, we have
\[ \left\|(\mathcal{W}^{(-1/2)}\hadam\mathcal{Z}_{\Omega})_{(1)}\right\|\leq 2\sigma\mu_1\sqrt{\log(d_1+\prod\limits_{j\neq 1}d_j)}.
\]

Similarly, we have
\begin{eqnarray*}
&&\left\|(\mathcal{W}^{(-1/2)}\hadam\mathcal{Y}_\Omega-\mathcal{W}^{(1/2)}\hadam \mathcal{T})_{(k)}\right\|_2\\
&\leq&\sqrt{r_k}\left\|\mathcal{T}\right\|_{\infty}\left\|(\mathcal{W}^{(-1/2)}\hadam \boldsymbol{1}_{\Omega}-\mathcal{W}^{(1/2)})_{(k)}\right\|_2+\left\|(\mathcal{W}^{(-1/2)}\hadam\mathcal{Z}_{\Omega})_{(k)}\right\|_2,
\end{eqnarray*}
with 
\[ \left\|(\mathcal{W}^{(-1/2)}\hadam\mathcal{Z}_{\Omega})_{(k)}\right\|_2\leq 2\sigma\mu_k\sqrt{\log(d_k+\prod\limits_{j\neq k}d_j)} 
\]with probability at least $1-\frac{1}{d_k+\prod_{j\neq k}d_j}$, for $k=2,\cdots,n$.

Plugging all these into \eqref{eqn:ube}, we can obtain the bound in our theorem.
\end{proof}

Next we are going to study the special case when the sampling set $\Omega\sim\mathcal{W}$.

\subsection{Case Study: \texorpdfstring{$\Omega\sim\mathcal{W}$}{Lg}}
\label{appendixB2} 
In this section, we would provide upper and lower bounds for the weighted\linebreak \mbox{HOSVD~algorithm.}
\subsubsection{Upper Bound }
First, let us understand the bounds $\lambda_{\ell}$ and $\mu_{\ell}$ in the case when $\Omega\sim\mathcal{W}$ for  $\ell=1,\cdots,n$.
\begin{Lemma}\label{lmm:ub for lmb_mu}
Let $\mathcal{W}=\boldsymbol{w_1}\out \cdots\out \boldsymbol{w_n}\in\mathbb{R}^{d_1\times \cdots\times d_n}$ be a CP rank-1 tensor so that all $(i_1,\cdots,i_n)\in[d_1]\times\cdots\times [d_n]$ with $\mathcal{W}_{i_1\cdots i_n}\in\left[\frac{1}{\sqrt{\prod_{j=1}^{n}d_j}},1\right]$. Suppose that $\Omega\subseteq[d_1]\times\cdots\times[d_n]$ so that for each $i_1\in[d_1], \cdots, i_n\in[d_n]$, $(i_1,\cdots,i_n)\in\Omega$ with probability $\mathcal{W}_{i_1\cdots i_n}$, independently for each $(i_1,\cdots,i_n)$. Then with probability at least $1-\sum_{\ell=1}^{n}\frac{2}{d_\ell+\prod_{j\neq \ell}d_j}$ 
over the choice of $\Omega$, we have for $\ell=1,\cdots,n$
\begin{equation}\label{eqn:lambda}
\lambda_\ell=\left\|(\mathcal W^{(1/2)}-\mathcal W^{(-1/2)}\hadam \boldsymbol{1}_\Omega)_{(\ell)}\right\|_2\leq2\sqrt{d_\ell+\prod\limits_{k\neq\ell}d_k}\log\left(d_\ell+\prod\limits_{k\neq\ell}d_k\right), 
\end{equation}
and
\begin{equation}\label{eqn:mu}
\mu_\ell\leq2\sqrt{\left(d_\ell+\prod\limits_{k\neq\ell}d_k\right)\log\left(d_\ell+\prod\limits_{k\neq\ell}d_k\right)}.
\end{equation}
\end{Lemma}
\begin{proof}
Fix $i_1\in[d_1]$. Bernstein's inequality yields
\begin{eqnarray*}
&&\mathbb{P}\left\{\sum\limits_{i_2,\cdots,i_n}\frac{1_{(i_1,\cdots,i_n)\in\Omega}}{\boldsymbol{w_{1}}(i_1)\cdots \boldsymbol{w_n}(i_n)}-\prod\limits_{k\neq1}d_k\geq t \right\}\\
&\leq&\exp\left(\frac{-t^2/2}{\sum\limits_{i_2,\cdots,i_n}\left(\frac{1}{\boldsymbol{w_{1}}(i_1)\cdots \boldsymbol{w_n}(i_n)}-1\right)+\frac{1}{3}\sqrt{\prod\limits_{k=1}^n d_k}t }\right).
\end{eqnarray*}
and 
\begin{eqnarray*}
&&\mathbb{P}\left\{\sum_{i_1}\frac{1_{(i_1,\cdots,i_n)\in\Omega}}{\boldsymbol{w_{1}}(i_1)\cdots \boldsymbol{w_n}(i_n)}-d_1\geq t \right\}\\
&\leq&\exp\left(\frac{-t^2/2}{\sum\limits_{i_1}\left(1/(\boldsymbol{w_{1}}(i_1)\cdots \boldsymbol{w_n}(i_n))-1\right)+\frac{1}{3}\sqrt{\prod\limits_{k=1}^n d_k}t }\right).
\end{eqnarray*}

Set $t=2\sqrt{2}(d_1+\prod\limits_{j\neq 1}d_j)\log(d_1+\prod\limits_{j\neq 1}d_j)$, then we have
\begin{eqnarray*}
&&\mathbb{P}\left\{\sum_{i_2,\cdots,i_n}\frac{1_{(i_1,i_2,\cdots,i_n)\in\Omega}}{\boldsymbol{w_{1}}(i_1)\cdots \boldsymbol{w_n}(i_n)}-\prod\limits_{k\neq 1}d_k\geq 2\sqrt{2}\left(d_1+\prod\limits_{j\neq 1}d_j\right)\log\left(d_1+\prod\limits_{j\neq 1}d_j\right) \right\}\\
&\leq& 1\Bigg/ \left(d_1+\prod\limits_{j\neq 1}d_j\right)^2
\end{eqnarray*}
 and 
 \begin{eqnarray*}
&&\mathbb{P}\left\{\sum_{i_1}\frac{1_{(i_1,i_2,\cdots,i_n)\in\Omega}}{\boldsymbol{w}_{1}(i_1)\boldsymbol{w}_{2}(i_2)\cdots \boldsymbol{w_n}(i_n)}-d_1\geq 2\sqrt{2}\left(d_1+\prod\limits_{j\neq 1}d_j\right)\log\left(d_1+\prod\limits_{j\neq 1}d_j\right) \right\}\\
&\leq&1\Bigg/\left(d_1+\prod_{j\neq 1}d_j\right)^2.
\end{eqnarray*}

Hence, by taking a union bound,
\begin{eqnarray*}
&&\mathbb{P}\left\{\max\left\{\max_{i_1}\sum_{i_2,\cdots,i_n}\frac{1_{(i_1,i_2,\cdots,i_n)\in\Omega}}{\boldsymbol{w}_{1}(i_1)\boldsymbol{w}_{2}(i_2)\cdots \boldsymbol{w_n}(i_n)},\max_{i_2,\cdots,i_n}\sum_{i_1}\frac{1_{(i_1,i_2,\cdots,i_n)\in\Omega}}{\boldsymbol{w}_{1}(i_1)\boldsymbol{w}_{2}(i_2)\cdots \boldsymbol{w_n}(i_n)}\right\}\right.\\
&&\left.\geq 4\left(d_1+\prod\limits_{j\neq 1}d_j\right)\log\left(d_1+\prod\limits_{j\neq 1}d_j\right) \right\}
\leq 
\frac{1}{d_1+\prod\limits_{j\neq 1}d_j}.
\end{eqnarray*}

Similarly, we have
\begin{eqnarray*}
\mathbb{P}\left\{\mu_k^2\geq 4\left(d_k+\prod\limits_{j\neq k}d_j\right)\log\left(d_k+\prod\limits_{j\neq k}d_j\right) \right\}&\leq&\frac{1}{d_k+\prod\limits_{j\neq k}d_j}, \text{ for all }k=2,\cdots,n.
\end{eqnarray*} 

Combining all these inequalities above, with probability at least $1-\sum_{\ell=1}^{n}\frac{1}{d_\ell+\prod_{j\neq \ell}d_j}$, \mbox{we have}
\[\mu_\ell\leq 2\sqrt{\left(d_\ell+\prod\limits_{k\neq\ell}d_k\right)\log\left(d_\ell+\prod\limits_{k\neq\ell}d_k\right)}, \text{ for all }\ell=1,\cdots,n.
\]

Next we would bound $\lambda_\ell$ in \eqref{eqn:lambda}. First of all, let's consider 
$\|(\mathcal{W}^{(1/2)}-\mathcal{W}^{(-1/2)}\hadam \boldsymbol{1}_\Omega)_{(1)}\|_2$. Set $\gamma_{i_1\cdots i_n}=\frac{\mathcal{W}_{i_1\cdots i_n}-1_{(i_1,\cdots,i_n)\in\Omega}}{\sqrt{\mathcal{W}_{i_1\cdots i_n}}}$. Then 
\[\left(\mathcal{W}^{(1/2)}-\mathcal{W}^{(-1/2)}\hadam \boldsymbol{1}_{\Omega}\right)_{(1)}=\sum_{i_1,\cdots,i_n}\gamma_{i_1\cdots i_n}\boldsymbol{\boldsymbol{e_{i_1}}}(\boldsymbol{e_{i_2}}\kron\cdots\kron \boldsymbol{e_{i_n}})^T.
\]

Notice that
\begin{eqnarray*}
&&\sum_{i_1,\cdots,i_n}\mathbb{E}\left(\gamma_{i_1\cdots i_n}^2\boldsymbol{\boldsymbol{e_{i_1}}}(\boldsymbol{e_{i_2}}\kron\cdots\kron \boldsymbol{e_{i_n}})^T(\boldsymbol{e_{i_2}}\kron\cdots\kron \boldsymbol{e_{i_n}})\boldsymbol{\boldsymbol{e_{i_1}}}^T\right)\\
&=&\sum_{i_1}\left(\sum_{i_2,\cdots,i_n}\mathbb{E}(\gamma_{i_1\cdots i_n}^2)\right)\boldsymbol{\boldsymbol{e_{i_1}}}\boldsymbol{\boldsymbol{e_{i_1}}}^T.
\end{eqnarray*}

Since $\mathbb{E}(\gamma_{i_1\cdots i_n}^2)=1-\mathcal{W}_{i_1\cdots i_n}\leq 1-\frac{1}{\sqrt{d_1\cdots d_n}}\leq 1$, then \[\left\|\sum_{i_1,\cdots,i_n}\mathbb{E}(\gamma_{i_1\cdots i_n}^2\boldsymbol{e_{i_1}}(\boldsymbol{e_{i_2}}\kron\cdots\kron \boldsymbol{e_{i_n}})^T(\boldsymbol{e_{i_2}}\kron\cdots\kron \boldsymbol{e_{i_n}})\boldsymbol{e_{i_1}}^T)\right\|_2\leq \prod\limits_{j\neq 1}d_j.
\]

Similarly, \[\left\|\sum_{i_1,\cdots,i_n}\mathbb{E}(\gamma_{i_1\cdots i_n}^2(\boldsymbol{e_{i_2}}\kron\cdots\kron \boldsymbol{e_{i_n}})\boldsymbol{e_{i_1}}^T\boldsymbol{e_{i_1}}(\boldsymbol{e_{i_2}}\kron\cdots\kron \boldsymbol{e_{i_n}})^T)\right\|_2\leq d_1.\]

In addition,
\[\left\|\gamma_{i_1\cdots i_n}\boldsymbol{e_{i_1}}(\boldsymbol{e_{i_2}}\kron\cdots\kron \boldsymbol{e_{i_n}})^T\right\|_2\leq\left(\prod\limits_{j=1}^n d_j\right)^{1/4}\leq\sqrt{\frac{d_1+\prod_{j\neq 1}d_j}{2}}.
\]

Then,  the matrix Bernstein Inequality (\cite{tropp2012user} Theorem 1.4) gives
\begin{align*}
&~\mathbb{P}\left\{ \left\|\left(\mathcal{W}^{(1/2)}-\mathcal{W}^{(-1/2)}\hadam \boldsymbol{1}_\Omega\right)_{(1)}\right\|_2\geq t\right\}\\
\leq&~\left(d_1+\prod\limits_{j\neq 1}d_j\right)\exp\left(-\frac{t^2/2}{\left(d_1+\prod\limits_{j\neq 1}d_j\right)+\frac{t}{3}\sqrt{\left(d_1+\prod\limits_{j\neq 1}d_j\right)\Bigg/2} } \right).
\end{align*}

Let $t=2\sqrt{d_1+\prod_{j\neq 1}d_j}\log\left(d_1+\prod_{j\neq 1}d_j\right)$, then we have
\begin{align*}
\mathbb{P}\left\{ \left\|\left(\mathcal{W}^{(1/2)}-\mathcal{W}^{(-1/2)}\hadam \boldsymbol{1}_{\Omega}\right)_{(1)}\right\|_2\geq 2\sqrt{d_1+\prod\limits_{j\neq 1}d_j}\log\left(d_1+\prod\limits_{j\neq 1}d_j\right)\right\}\leq\frac{1}{d_1+\prod\limits_{j\neq 1}d_j}.
\end{align*}

Similarly, \begin{align*}
\mathbb{P}\left\{ \left\|\left(\mathcal{W}^{(1/2)}-\mathcal{W}^{\left(-1/2\right)}\hadam \boldsymbol{1}_{\Omega}\right)_{(k)}\right\|_2 
\geq 2\sqrt{d_k+\prod_{j\neq k}d_k}\log\left(d_k+\prod\limits_{j\neq k}d_j\right)\right\}\leq\frac{1}{d_k+\prod\limits_{j\neq k}d_j},
\end{align*}
for all $k=2,\cdots,n$.

Thus, with probability at least $1-\sum_{\ell=1}^n\frac{1}{d_\ell+\prod_{j\neq \ell}d_j}$, we have
\[
\left\|(\mathcal W^{(1/2)}-\mathcal W^{(-1/2)}\hadam \boldsymbol{1}_{\Omega})_{(\ell)}\right\|_2\leq2\sqrt{d_\ell+\prod\limits_{k\neq\ell}d_k}\log\left(d_\ell+\prod\limits_{k\neq\ell}d_k\right) , \text{ for all } \ell=1,\cdots,n.
\]

By a union of bounds in \eqref{eqn:mu} and \eqref{eqn:lambda}, we could establish the lemma. 
\end{proof}

\begin{Lemma}\label{lmm:est_omega}
Let $m=\left\|\mathcal{W}^{(1/2)}\right\|_F^2$. Then with probability at least $1-2\exp(-3m/104)$, over the choice of $\Omega$
\[||\Omega|-m|\leq \frac{m}{4}.
\]
\end{Lemma}
\begin{proof}
Please note that
\[||\Omega|-m|=\left|\sum_{i_1,\cdots,i_n}(1_{(i_1,\cdots,i_n)\in\Omega}-\mathcal W_{i_1\cdots i_n})\right|=\left|\sum_{i_1,\cdots,i_n}(1_{(i_1,\cdots,i_n)\in\Omega}-\mathbb{E}(1_{(i_1,\cdots,i_n)\in\Omega})\right|,
\]
which is the sum of zero-mean independent random variables. Observe that $|1_{(i_1,\cdots,i_n)\in\Omega}-\mathbb{E}(1_{(i_1,\cdots,i_n)\in\Omega})|=|1_{(i_1,\cdots,i_n)\in\Omega}-\mathcal W_{i_1\cdots i_n}|\leq 1$ and \[\sum_{i_1,\cdots,i_n}\mathbb{E}(1_{(i_1,\cdots,i_n)\in\Omega}-\mathcal{W}_{i_1\cdots i_n})^2=\sum_{i_1,\cdots,i_n}(\mathcal{W}_{i_1\cdots i_n}-\mathcal{W}_{i_1\cdots i_n}^2)\leq m.
\]

By Bernstein's inequality,
\begin{eqnarray*}
\mathbb{P}\left(||\Omega|-m|\geq t \right)&\leq& 2\exp\left(-\frac{t^2/2}{m+t/3} \right).
\end{eqnarray*}

Set $t=m/4$, then we have
\begin{eqnarray*}
\mathbb{P}\left(||\Omega|-m|\geq m/4 \right)&\leq& 2\exp\left(-\frac{m^2/32}{m+m/12} \right)=2\exp(-3m/104).
\end{eqnarray*}
\end{proof}

Next let us give the formal statement for the upper bounds in Theorem \ref{thm:bfsomega}.
 \begin{theorem}\label{thm: omgsw}
 
Let $\mathcal W=\boldsymbol{w_1}\out \cdots \out \boldsymbol{w_n}\in\mathbb{R}^{d_1\times \cdots\times d_n}$ be a CP rank-1 tensor so that for all $(i_1,\cdots,i_n)\in[d_1]\times\cdots\times[d_n]$ we have $\mathcal W_{i_1\cdots i_n}\in\left[\frac{1}{\sqrt{d_1\cdots d_n}},1\right]$. Suppose that we choose each $(i_1,\cdots,i_n)\in[d_1]\times\cdots\times[d_n]$ independently with probability $\mathcal{W}_{i_1\cdots i_n}$ to form a set $\Omega\subseteq[d_1]\times\cdots\times[d_n]$.  Then with probability at least
$$1-2\exp\left(-\frac{3}{104}\sqrt{\prod_{j=1}^nd_j}~\right)-\sum_{k=1}^{n}\frac{2}{d_k+\prod_{j\neq k}d_j} $$

For the weighted HOSVD Algorithm named  $\mathcal{A}$, $\mathcal{A}$ returns $\widehat{\mathcal{T}}=\mathcal{A}(\mathcal{T}_\Omega+\mathcal{Z}_\Omega)$ for any Tucker $\text{rank}$ $\boldsymbol{r}$ tensor $\mathcal{T}$ with $\left\|\mathcal{T}\right\|_\infty\leq\beta$ so that with probability at least $1-\sum_{k=1}^n\frac{1}{d_k+\prod_{j\neq k}d_j}$ over the choice of $\mathcal{Z}$,
\begin{eqnarray*}\label{uppb}
    \frac{\left\|\mathcal{W}^{(1/2)}\hadam(\mathcal{T}-\widehat{\mathcal{T}})\right\|_F}{\left\|\mathcal{W}^{(1/2)}\right\|_F}
   &\leq&\frac{\sqrt{5}\beta}{\sqrt{|\Omega|}}\left( \sum_{k=1}^n 3r_k\sqrt{d_k+\prod\limits_{j\neq k}d_j}\log\left(d_k+\prod\limits_{j\neq k}d_j\right)\right)\nonumber\\
   &&+\frac{\sqrt{5}\sigma}{|\Omega|}\left(\sum_{k=1}^n6\sqrt{r_k(d_k+\prod\limits_{j\neq k}d_j)}\log\left(d_k+\prod\limits_{j\neq k}d_j\right)\right)\nonumber
\end{eqnarray*}
\end{theorem}
\begin{proof}
This is directly from Theorem \ref{thm:ub low-rank tensor}, Lemmas \ref{lmm:ub for lmb_mu} and \ref{lmm:est_omega}.
\end{proof}

\subsubsection{Lower Bound} 
To deduce the lower bound, we have to construct a finite  subset $S$ in the cone $K_{\boldsymbol{r}}$ so that we can approximate the minimal distance between two different elements in $S$. Before we prove the lower bound,  we  need the following theorems and lemmas.
\begin{theorem}[Hanson-Wright inequality]\label{thm:HWI} There is some constant $c > 0$ so that the following holds. Let
$\xi \in\{ 0,\pm 1\}^d$
be a vector with mean-zero, independent entries, and let $F$ be any matrix which has zero
diagonal. Then
\[\mathbb{P}\left\{
|\xi^T F \xi| > t\right\}\leq 2 \exp\left(
-c\cdot  \min\left\{\frac{t^2}{\|F\|_F^2}, \frac{t}{\|F\|_2}
\right\}
\right).
\]
\end{theorem}
\begin{theorem}[Fano's Inequality] Let $\mathcal{F}=\{f_0,\cdots,f_n\}$ be a collection of densities on $\mathcal{K}$, and suppose that $\mathcal{A}:\mathcal{K}\rightarrow\{0,\cdots,n\}$. Suppose there is some $\beta>0$ such that for any $i\neq j$, $D_{KL}(f_i\|f_j)\leq\beta$. Then
\[\max_{i}\mathbb{P}_{K\sim f_i}\left\{\mathcal{A}(K)\neq i \right\}\geq 1-\frac{\beta+\log(2)}{\log(n)}.
\]
\end{theorem}
The following lemma specializes Fano's Inequality to our setting, which is a generalization of (\cite{foucart2019weighted} Lemma 19).  In the following lemma, we show that for any reconstruction algorithm on a set $K\subseteq \mathbb{R}^{d_1\times\cdots\times d_n}$, with probability no less than $\frac{1}{2}$, there exists some elements in $K$ such that the weighted reconstruction error  is bounded below by some quantity, where the quantity  is independent of the algorithm. 
\begin{Lemma}\label{lmm:FanoE}
Let $K\subseteq \mathbb{R}^{d_1\times \cdots\times d_n}$, and let $S\subseteq K$ be a finite subset of $K$ so that $|S|>16$. Let $\Omega\subseteq [d_1]\times\cdots\times [d_n]$ be a sampling pattern. Let $\sigma>0$ and choose
\[\kappa\leq\frac{\sigma\sqrt{\log|S|}}{4\max_{\mathcal{T}\in S}\left\|\mathcal{T}_\Omega\right\|_F},
\]
and suppose that 
\[\kappa S\subseteq K.
\]

Let $\mathcal{Z}\in\mathbb{R}^{d_1\times \cdots\times d_n}$ be a tensor whose entries $\mathcal{Z}_{i_1\cdots i_n}$ are i.i.d., $\mathcal{Z}_{i_1\cdots i_n}\sim \mathcal{N}(0,\sigma^2)$. Let $\mathcal{H}\subseteq \mathbb{R}^{d_1\times\cdots\times d_n}$ be any weight tensor.

Then for any algorithm $\mathcal{A}:\mathbb{R}^{\Omega}\rightarrow \mathbb{R}^{d_1\times \cdots\times d_n}$ that takes as input $\mathcal{T}_\Omega+\mathcal{Z}_\Omega$ for $\mathcal{T}\in K$ and outputs an estimate $\widehat{\mathcal{T}}$ to $\mathcal{T}$, there is some $\mathcal{X}\in K$ so that
\begin{equation}\label{eqn:Fano}
\left\|\mathcal H\hadam(\mathcal{A}(\mathcal{X}_\Omega+\mathcal{Z}_{\Omega})-\mathcal{X})\right\|_F\geq \frac{\kappa}{2}\min_{\mathcal{T}\neq\mathcal{T}'\in S}\left\|\mathcal H\hadam(\mathcal{T}-\mathcal{T}')\right\|_F
\end{equation}
with probability at least $\frac{1}{2}$.
\end{Lemma}
\begin{proof}
Consider the set
\[S'=\kappa S=\{\kappa \mathcal{T}:\mathcal{T}\in S\}
\]
which is a scaled version of $S$. By our assumption, $S'\subseteq K$.

Recall that the Kullback--Leibler (KL) divergence between two multivariate Gaussians is given by
\begin{eqnarray*}
&&D_{KL}(\mathcal{N}(\boldsymbol{\mu}_1,\Sigma_1)\|\mathcal{N}(\boldsymbol{\mu}_2,\Sigma_2))\\
&=&\frac{1}{2}\left( \log\left(\frac{\det(\Sigma_2)}{\det(\Sigma_1)}\right)-n+tr(\Sigma_2^{-1}\Sigma_1)+\langle \Sigma_2^{-1}(\boldsymbol{\mu}_2-\boldsymbol{\mu}_1),\boldsymbol{\mu}_2-\boldsymbol{\mu}_1\rangle \right),
\end{eqnarray*}
where $\boldsymbol{\mu}$$_1$, $\boldsymbol{\mu}$$_2\in\mathbb{R}^n$. 

Specializing to $\mathcal{U},\mathcal{V}\in S'$, with $I=I_{\Omega\times\Omega}$
\begin{eqnarray*}
D_{KL}(\mathcal{U}_\Omega+\mathcal{Z}_\Omega\| \mathcal{V}_\Omega+\mathcal{Z}_\Omega)&=&D_{KL}(\mathcal{N}(\mathcal{U}_\Omega,\sigma^2 I)\| \mathcal{N}(\mathcal{V}_\Omega,\sigma^2 I))\\
&=&\frac{\left\|\mathcal{U}_\Omega-\mathcal{V}_\Omega\right\|_F^2}{2\sigma^2}\\
&\leq&\max_{\mathcal{T}\in S'}\frac{2\left\|\mathcal{T}_\Omega \right\|_F^2}{\sigma^2}
=\frac{2\kappa^2}{\sigma^2}\max_{\mathcal{T}\in S}\left\|\mathcal{T}_\Omega\right\|_F^2.
\end{eqnarray*}

Suppose that $\mathcal{A}$ is as in the statement of the lemma. Define an algorithm $\overline{\mathcal{A}}:\mathbb{R}^{\Omega}\rightarrow\mathbb{R}^{d_1\times \cdots\times d_n}$ so that for any $\mathcal{Y}\in\mathbb{R}^{\Omega}$ if there exists $\mathcal{T}\in S'$ such that  
\[ \left\|\mathcal{H}\hadam(\mathcal{T}-\mathcal{A}(\mathcal{Y}))\right\|_F<\frac{1}{2}\min_{\mathcal{T}\neq\mathcal{T}'\in S'}\|\mathcal{H}\hadam(\mathcal{T}-\mathcal{T}')\|_F:=\frac{\rho}{2},
\] then set $\overline{\mathcal{A}}(\mathcal{Y})=\mathcal{T}$ (notice that if such $\mathcal{T}$ exists, then it is  unique),  otherwise, set\linebreak $\overline{\mathcal A}(\mathcal{Y})=\mathcal{A}(\mathcal Y)$.

Then by the Fano's inequality, there is some $\mathcal{T}\in S'$ so that
\begin{eqnarray*}
\mathbb{P}\left\{\overline{\mathcal{A}}(\mathcal{T}_\Omega+\mathcal{Z}_\Omega)\neq \mathcal{T}
\right\}&\geq & 1-\frac{2\max_{\mathcal{T}\in S'}\|\mathcal{T}_\Omega\|_F^2}{\sigma^2\log(|S|-1)}-\frac{\log(2)}{\log(|S|-1)}\\
&=&1-\frac{2\kappa^2\max_{\mathcal{T}\in S}\|\mathcal{T}_\Omega\|_F^2}{\sigma^2\log(|S|-1)}-\frac{\log(2)}{\log(|S|-1)}\\
&\geq&1-\frac{1}{4}-\frac{1}{4}=\frac{1}{2}.
\end{eqnarray*}

If $\overline{\mathcal{A}}(\mathcal{T}_\Omega+\mathcal{Z}_\Omega)\neq \mathcal{T}$, then $\|\mathcal{H}\hadam(\mathcal{A}(\mathcal{T}_\Omega+\mathcal{Z}_\Omega)-\mathcal{T})\|_F>\rho/2$, and so
\[\mathbb{P}\left\{\|\mathcal{H}\hadam(\mathcal{A}(\mathcal{T}_\Omega+\mathcal{Z}_\Omega)-\mathcal{T})\|_F\geq\rho/2
\right\}\geq\mathbb{P}\left\{\overline{\mathcal{A}}(\mathcal{T}_\Omega+\mathcal{Z}_\Omega)\neq \mathcal{T}
\right\}\geq 1/2.
\]

Finally, we observe that
\[\frac{\rho}{2}=\frac{1}{2}\min_{\mathcal{T}\neq\mathcal{T}'\in S'}\|\mathcal{H}\hadam(\mathcal{T}-\mathcal{T}')\|_F=\frac{\kappa}{2}\min_{\mathcal{T}\neq\mathcal{T}'\in S}\|\mathcal{H}\hadam(\mathcal{T}-\mathcal{T}')\|_F,
\]
which completes the proof.
\end{proof}
To understand the lower bound $\frac{\kappa}{2}\min_{\mathcal{T}\neq\mathcal{T}\in S}\|\mathcal{H}\hadam(\mathcal{T}-\mathcal{T}')\|_F$ in  \eqref{eqn:Fano}, we construct a specific finite subset $S$ for the cone of Tucker rank $\boldsymbol{r}$ tensors in the following lemma.
\begin{Lemma}\label{lmm:lb}
There is some constant $c$ so that the following holds. Let $d_1,\cdots,d_n>0$ and $r_1,\cdots,r_n>0$ be sufficiently large.  Let $K$ be the cone of Tucker rank $ \boldsymbol{r}$ tensors with $\boldsymbol{r}=[r_1~\cdots~r_n]$, $\mathcal{H}$ be any CP rank-1 weight tensor, and  $\mathcal{B}$ be any CP rank-1 tensor with $\|\mathcal{B}\|_\infty\leq 1$. Write $\mathcal{H}=\boldsymbol{h_1}\out \cdots\out \boldsymbol{h_n}$ and $\mathcal{B}=\boldsymbol{b_1}\out \cdots\out \boldsymbol{b_n}$, and 
\[ \boldsymbol{w_1}=(\boldsymbol{h_1}\hadam \boldsymbol{b_1})^{(2)}, \cdots, \boldsymbol{w_n}=(\boldsymbol{h_n}\hadam \boldsymbol{b_n})^{(2)}.
\]

Let $$\gamma=\sqrt{\frac{1}{2}\left(\prod\limits_{k=1}^nr_k\right)\log\left(8\prod\limits_{k=1}^nd_k\right)}.$$

There is a set $S\subseteq K\cap\gamma \boldsymbol{B}_\infty$ so that
\begin{enumerate}
    \item The set has size $|S|\geq N$, for
    \begin{multline*}
    N=C\exp\left({c\cdot\min\left\{ \frac{\prod\limits_{k=1}^nr_k}{\left(\prod\limits_{k=1}^n(2r_k(\|\boldsymbol{w}_k\|_2/\|\boldsymbol{w}_k\|_1)^2+1)\right)-1},\prod\limits_{k=1}^n r_k,\right.}\right.\\
 \left. \left.\frac{\prod\limits_{k=1}^nr_k}{\left(\prod\limits_{k=1}^n(2\|\boldsymbol{w_k}\|_2/\|\boldsymbol{w}_k\|_1\sqrt{r_k\log(r_k)}+2\|\boldsymbol{w_k}\|_\infty/\|\boldsymbol{w}_k\|_1 r_k\log(r_k)+1)\right)-1}\right\} \right).
    \end{multline*}
    \item $\|\mathcal{T}_{\Omega}\|_F\leq 2\sqrt{\prod\limits_{k=1}^nr_k}\|\mathcal{B}_{\Omega}\|_F$ for all $\mathcal{T}\in S$.
    \item $\left\|\mathcal {H}\hadam(\mathcal{T}-\widetilde{\mathcal{T}})\right\|_F\geq \sqrt{\prod\limits_{k=1}^nr_k}\left\|\mathcal{H}\hadam \mathcal{B}\right\|_F$ for all $\mathcal{T}\neq\widetilde{\mathcal{T}} \in S$.
\end{enumerate}
\end{Lemma}
\begin{proof}
Let $\Psi\subseteq\{\pm1\}^{r_1\times\cdots\times  r_n}$ be a set of random $\pm1$-valued tensors chosen uniformly at random with replacement, of size $4N$. Choose
${}^i{U}\in\{\pm1\}^{d_i\times r_i}$ to be determined below for all $i=1,\cdots,n$ .

Let \[
S=\left\{\mathcal{B}\hadam(\mathcal{C}\times_1 {}^1U\times_2\cdots\times_n{}^nU):\mathcal{C}\in \Psi \right\}.
\]

First of  all,  we would estimate $\left\|\mathcal{T}_{\Omega}\right\|_F$ and $\left\|\mathcal{T}\right\|_{\infty}$. Please note that
\begin{eqnarray*}
\mathbb{E}\left\|\mathcal{T}_{\Omega}\right\|_F^2&=&\mathbb{E}\sum_{(i_1,\cdots,i_n)\in\Omega}\mathcal{B}_{i_1\cdots i_n}^2\left(\sum_{j_1,\cdots,j_n}\mathcal{C}_{j_1\cdots j_n}{}^1U(i_1,j_1)\cdots {}^nU(i_n,j_n)\right)^2=\left(\prod\limits_{i=1}^n r_i\right)\|\mathcal{B}_{\Omega}\|_F^2,
\end{eqnarray*}
where the expectation is over the random choice of $\mathcal{C}$. Then by Markov's inequality,
\[\mathbb{P}\left\{\|\mathcal{T}_{\Omega}\|_F^2\geq \left(4\prod\limits_{i=1}^n r_i\right)\|\mathcal{B}_{\Omega}\|_F^2\right\}\leq \frac{1}{4}.
\]

We also have
\[\|\mathcal{T}\|_{\infty}=\max_{i_1,\cdots,i_n}|\mathcal{B}_{i_1\cdots i_n}|\left|\sum_{j_1,\cdots,j_n}\mathcal{C}_{j_1\cdots j_n}{}^1U(i_1,j_1)\cdots {}^nU(i_n,j_n)\right|.
\]

By Hoeffding's inequality, we have
\begin{eqnarray*}
\mathbb{P}\left\{\left|\sum_{j_1,\cdots,j_n}\mathcal{C}_{j_1\cdots j_n}{}^1U(i_1,j_1)\cdots {}^nU(i_n,j_n)\right|\geq t\right\}&\leq&2\exp\left( -\frac{2t^2}{\prod_{k=1}^n r_k}\right).
\end{eqnarray*}

Using the fact that $|\mathcal{B}_{i_1\cdots i_n}|\leq 1$ and a union bound over all $\prod\limits_{k=1}^nd_k$ values of $i_1,\cdots,i_n$, we conclude that
\begin{eqnarray*}
&&\mathbb{P}\left\{\|\mathcal{T}\|_{\infty}\geq\sqrt{\frac{1}{2}\left(\prod\limits_{k=1}^n r_k\right)\log\left(8\prod\limits_{k=1}^nd_k\right)}\right\}\\
&\leq& 
\left(\prod\limits_{k=1}^nd_k\right)\mathbb{P}\left\{ \left|\sum_{j_1,\cdots,j_n}\mathcal{C}_{j_1\cdots j_n}{}^1U(i_1,j_1)\cdots {}^nU(i_n,j_n)\right|\geq  \sqrt{\frac{1}{2}\left(\prod\limits_{k=1}^nr_k\right)\log\left(8\prod\limits_{k=1}^nd_k\right)}\right\}\\
&\leq& \frac{1}{4}.
\end{eqnarray*}

Thus, for a tensor $\mathcal{T}\in S$, the probability that both of $\|\mathcal{T}\|_{\infty}\leq\sqrt{\frac{1}{2}\left(\prod\limits_{k=1}^nr_k\right)\log\left(8\prod\limits_{k=1}^nd_k\right)}$ and $\|\mathcal{T}_{\Omega}\|_F\leq 2\sqrt{\prod\limits_{k=1}^nr_k}\|\mathcal{B}_{\Omega}\|_F$ hold is at least $\frac{1}{2}$. {Thus, by a Chernoff bound it follows that with  probability at least $1-\exp(-CN)$ for some constant $C$, there are at least $\frac{|S|}{4}$ tensors $\mathcal{T}\in S$ such that all of these hold. Let $\widetilde{S}\subseteq S$ be the set of such $\mathcal{T}$'s. The set guaranteed in the statement of the lemma will be $\widetilde{S}$, which satisfies both item 1 and 2 in the lemma and is also contained in $K\cap\gamma\boldsymbol{B}_{\infty}$.}

Thus, we consider item 3: we are going to show that this holds for $S$ with high probability, thus in particularly it will hold for $\widetilde{S}$, and this will complete the proof of \mbox{the lemma.}

Fix $\mathcal{T}\neq\widetilde{\mathcal{T}}\in S$, and write
\begin{eqnarray*}
&&\left\|\mathcal H\hadam(\mathcal{T}-\widetilde{\mathcal{T}})\right\|_F^2\\
&=&\left\|\mathcal H\hadam\mathcal{B}\hadam((\mathcal{C}-\widetilde{\mathcal{C}})\times_1 {}^1 U\times_2\cdots\times_n{}^nU)\right\|_F^2\\
&=&\sum_{i_1,\cdots,i_n}\mathcal{H}_{i_1\cdots i_n}^2\mathcal{B}_{i_1\cdots i_n}^2\left(\sum_{j_1,\cdots,j_n}(\mathcal{C}_{j_1\cdots j_n}-\widetilde{\mathcal{C}}_{j_1\cdots j_n}){}^1U(i_1,j_1)\cdots {}^nU(i_n,j_n)\right)^2\\
&=&4\sum_{i_1,\cdots,i_n}\mathcal{H}_{i_1\cdots i_n}^2 \mathcal{B}_{i_1\cdots i_n}^2\left\langle \boldsymbol{\xi},{}^1U(i_1,:)\kron \cdots\kron {}^nU(i_n,:)\right\rangle^2,
\end{eqnarray*}
where $\boldsymbol{\xi}$ is the vectorization of $\frac{1}{2}(\mathcal{C}-\widetilde{\mathcal{C}})$. Thus, each entry of $\boldsymbol{\xi}$ is independently $0$ with probability $\frac{1}{2}$ or $\pm1$ with probability $\frac{1}{4}$ each. Rearranging the terms, we have
\begin{eqnarray}\label{eqn:lb}
\left\|\mathcal H\hadam(\mathcal{T}-\widetilde{\mathcal{T}})\right\|_F^2&=&4\boldsymbol{\xi}^T\left({}^1U\kron\cdots\kron{}^nU\right)^T \left(D_1\kron\cdots\kron D_n\right)\left({}^1U\kron\cdots\kron {}^nU\right)\boldsymbol{\xi}\nonumber\\
&=&4\boldsymbol{\xi}^T\left(\left({}^1U^TD_1{}^1U\right)\kron\cdots\kron\left({}^nU^TD_n{}^nU\right) \right)\boldsymbol{\xi}\nonumber\\
&=&4\boldsymbol{\xi}^T\left(\kron_{k=1}^n\left({}^kU^TD_k{}^kU\right) \right)\boldsymbol{\xi},
\end{eqnarray}
where $D_k$ denotes the $ d_k\times d_k$ diagonal matrix with $\boldsymbol{w_k}$ on the diagonal.

To understand \eqref{eqn:lb}, we need to understand the matrix $\kron_{k=1}^n\left({}^kU^TD_k{}^kU\right) \in\mathbb{R}^{\prod\limits_{k=1}^nr_k \times \prod\limits_{k=1}^nr_k}$. 
The diagonal of this matrix is $\left(\prod\limits_{k=1}^n\left\|\boldsymbol{w_k}\right\|_1\right)I$. We will choose the matrix ${}^kU$ for $k=1,\cdots,n$ so that the off-diagonal terms are small.
\end{proof}
{
\begin{claim}\label{cl: U}
There are matrices ${}^kU\in\{\pm 1\}^{d_k\times r_k}$ for $k=1,\cdots,n$ such that:
\begin{enumerate}
    \item [(a)]
    \begin{eqnarray*}
    \left\|\left(\kron_{k=1}^n\left({}^kU^TD_k{}^kU\right)\right) -\left(\prod\limits_{j=1}^n\left\|\boldsymbol{w_j}\right\|_1\right) I\right\|_F^2\leq\left(\prod\limits_{k=1}^n\left(2r_k^2\|\boldsymbol{w_k}\|_2^2+r_k\|\boldsymbol{w_k}\|_1^2\right)\right)-\prod\limits_{k=1}^n\left(r_k\|\boldsymbol{w_k}\|_1^2\right).
    \end{eqnarray*}
    \item [(b)] 
    \begin{eqnarray*}
    &&\left\|\left(\kron_{k=1}^n({}^kU^TD_k{}^kU)\right)-\left(\prod\limits_{j=1}^n\|\boldsymbol{w_j}\|_1\right) I\right\|_2\nonumber\\
    &\leq& \small{\max\left\{\prod\limits_{k=1}^n(2\left\|\boldsymbol{w_k}\right\|_2\sqrt{r_k\log(r_k)}+2\left\|\boldsymbol{w_k}\right\|_\infty r_k\log(r_k)+\left\|\boldsymbol{w_k}\right\|_1)-\prod\limits_{k=1}^n\left\|\boldsymbol{w_k}\right\|_1, \prod\limits_{k=1}^n\left\|\boldsymbol{w_k}\right\|_1\right\}}.
    \end{eqnarray*}
\end{enumerate}
\end{claim}
\begin{proof}
By (\cite{foucart2019weighted} Claim 22), there exist matrices ${}^kU\in\{\pm 1\}^{d_k\times r_k}$ such that:
\begin{enumerate}
    \item [(a)]$\left\|{}^kU^TD_k{}^kU\right\|_F^2\leq 2r_k^2\left\|\boldsymbol{w_k}\right\|_2^2+r_k\left\|\boldsymbol{w_k}\right\|_1^2$ and
     \item [(b)] $\left\|{}^kU^TD_k{}^kU\right\|_2\leq2\left\|\boldsymbol{w_k}\right\|_2\sqrt{r_k\log(r_k)}+2\left\|\boldsymbol{w_k}\right\|_\infty r_k\log(r_k)+\|\boldsymbol{w_k}\|_1 $.
\end{enumerate}

Using (a) and the fact that $\left\|\kron_{k=1}^n({}^kU^TD_k{}^kU)\right\|_F^2=\prod\limits_{k=1}^n\left\|{}^kU^TD_k{}^kU\right\|_F^2$,  we have  \begin{eqnarray*}
&&\left\|\left(\kron_{k=1}^n\left({}^kU^TD_k{}^kU\right)\right) -\left(\prod\limits_{k=1}^n\|\boldsymbol{w_j}\|_1\right) I\right\|_F^2\\
&=&\left\|\kron_{k=1}^n\left({}^kU^TD_k{}^kU\right)\right\|_F^2-\left\| \left(\prod\limits_{k=1}^n\|\boldsymbol{w_j}\|_1\right) I\right\|_F^2\\
&\leq& \left(\prod\limits_{k=1}^n\left(2r_k^2\|\boldsymbol{w_k}\|_2^2+r_k\|\boldsymbol{w_k}\|_1^2\right)\right)-\prod\limits_{k=1}^n\left(r_k\|\boldsymbol{w_k}\|_1^2\right).
\end{eqnarray*}

By (b) and the fact that $\left\|\kron_{k=1}^n({}^kU^TD_k{}^kU)\right\|_2=\prod\limits_{k=1}^n\left\|{}^kU^TD_k{}^kU\right\|_2$ (see \cite{LF1972}), we have 
\begin{align*}
&~\left\|\left(\kron_{k=1}^n\left({}^kU^TD_k{}^kU\right)\right)- \left(\prod\limits_{k=1}^n \|\boldsymbol{w_k}\|_1\right)I\right\|_2\\
\leq~& \max\left\{\prod\limits_{k=1}^n\|\boldsymbol{w_k}\|_1,\right.
\left.\left(\prod\limits_{k=1}^n\left(2\|\boldsymbol{w_k}\|_2\sqrt{r_k\log(r_k)}+2\|\boldsymbol{w_k}\|_\infty r_k\log(r_k)+\|\boldsymbol{w_k}\|_1\right)\right)-\prod\limits_{k=1}^n\|\boldsymbol{w_k}\|_1\right\}.
\end{align*}
\end{proof}
}
Having chosen  matrices ${}^kU$ for $k=1,\cdots,n$, we can now analyze the expression~\eqref{eqn:lb}.
\begin{claim}
 There are constants $c, c'$
so that with probability at least

\begin{eqnarray*}
&&1 -2\exp\left(-c''\prod\limits_{k=1}^nr_k\right)-2\exp\left(-c'\cdot\min\left\{ \frac{\prod\limits_{k=1}^n(r_k\|\boldsymbol{w}_k\|_1^2)}{\prod\limits_{k=1}^n(2r_k\|\boldsymbol{w}_k\|_2^2+\|\boldsymbol{w}_k\|_1^2)-\prod\limits_{k=1}^n\|\boldsymbol{w}_k\|_1^2},\right.\right.\\
&&\left.\left.\prod\limits_{k=1}^n r_k,\frac{\prod\limits_{k=1}^n(r_k\|\boldsymbol{w}_k\|_1)}{\left(\prod\limits_{k=1}^n(2\|\boldsymbol{w_k}\|_2\sqrt{r_k\log(r_k)}+2\|\boldsymbol{w_k}\|_\infty r_k\log(r_k)+\|\boldsymbol{w_k}\|_1)\right)-\prod\limits_{k=1}^n\|\boldsymbol{w_k}\|_1}
\right\} \right),
\end{eqnarray*}
we have
\[\left\|\mathcal{H}\hadam\left(\mathcal{T}-\widetilde{\mathcal{T}}\right)\prod\limits_{k=1}^n\left\|\boldsymbol{w_k}\right\|_1\right\|_F^2\geq \prod\limits_{k=1}^n\left(r_k\|\boldsymbol{w_k}\|_1\right).
\]
\end{claim}
\begin{proof}
 We break $\left\|\mathcal{H}\hadam(\mathcal{T}-\widetilde{\mathcal{T}})\right\|_F^2$ into two terms:
\begin{eqnarray*}
&&\left\|\mathcal{H}\hadam(\mathcal{T}-\widetilde{\mathcal{T}})\right\|_F^2\\
&=&4\boldsymbol{\xi}^T\left(\kron_{k=1}^n{}^kU^TD_k{}^kU\right)\boldsymbol{\xi}\\
&=&4\boldsymbol{\xi}^T\left(\kron_{k=1}^n\left({}^kU^TD_k{}^kU\right)-\left(\prod\limits_{k=1}^n\|\boldsymbol{w_k}\|_1\right)I\right)\boldsymbol{\xi}
+4\left(\prod\limits_{k=1}^n\|\boldsymbol{w_k}\|_1\right)\boldsymbol{\xi}^T\boldsymbol{\xi}\\
&:=&(I)+(II).
\end{eqnarray*}

For the first term (I), we will use the Hanson-Wright Inequality (see Theorem \ref{thm:HWI}). In our case, the matrix
$F=4\left(\kron_{k=1}^n\left({}^kU^TD_k{}^kU\right)-\left(\prod\limits_{k=1}^n\|\boldsymbol{w_k}\|_1\right)I\right)$. The Frobenius norm of this matrix is bounded by
\begin{eqnarray*}
\|F\|_F^2&\leq& 16\left(\prod\limits_{k=1}^n\left(2r_k^2\|\boldsymbol{w_k}\|_2^2+r_k\|\boldsymbol{w_k}\|_1^2\right)-\prod\limits_{k=1}^n\left(r_k\|\boldsymbol{w_k}\|_1^2\right)\right).
\end{eqnarray*}

The operator norm of $F$ is bounded by
\begin{eqnarray*}
&&\|F\|_2\\
&\leq&4\max\left\{\prod\limits_{k=1}^n(2\|\boldsymbol{w_k}\|_2\sqrt{r_k\log(r_k)}+2\|\boldsymbol{w_k}\|_\infty r_k\log(r_k)+\|\boldsymbol{w_k}\|_1)-\prod\limits_{k=1}^n\|\boldsymbol{w_k}\|_1, \prod\limits_{k=1}^n\|\boldsymbol{w_k}\|_1\right\}.
\end{eqnarray*}

Thus, the Hanson-Wright inequality implies that
\begin{eqnarray*}
&&\mathbb{P}\left\{(I)\geq t \right\}\\\
&\leq& 2\exp\left(-c\cdot\min\left\{\frac{t^2}{16\prod\limits_{k=1}^n\left(2r_k^2\|\boldsymbol{w_k}\|_2^2+r_k\|\boldsymbol{w_k}\|_1^2\right)-16\prod\limits_{k=1}^n\left(r_k\|\boldsymbol{w_k}\|_1^2\right)},\frac{t}{ 4\prod\limits_{k=1}^n\|\boldsymbol{w_k}\|_1},\right.\right.\\
&&  \left.\left.\frac{t}{4\left(\prod\limits_{k=1}^n(2\|\boldsymbol{w_k}\|_2\sqrt{r_k\log(r_k)}+2\|\boldsymbol{w_k}\|_\infty r_k\log(r_k)+\|\boldsymbol{w_k}\|_1)-\prod\limits_{k=1}^n\|\boldsymbol{w_k}\|_1\right)}
 \right\}\right).
\end{eqnarray*}

Plugging in $t =\frac{1}{2}\prod\limits_{k=1}^nr_k\|\boldsymbol{w_k}\|_1$, and replacing the constant $c$ with a different constant $c'$, we have
\begin{eqnarray}\label{eqn:(I)}
&&\mathbb{P}\left\{(I)\geq \frac{1}{2}\prod\limits_{k=1}^nr_k\|\boldsymbol{w_k}\|_1 \right\}\nonumber\\
&\leq& 2\exp\left(-c'\cdot\min\left\{ \frac{\prod\limits_{k=1}^nr_k}{\left(\prod\limits_{k=1}^n(2r_k(\|\boldsymbol{w}_k\|_2/\|\boldsymbol{w}_k\|_1)^2+1)\right)-1},\prod\limits_{k=1}^n r_k,\right.\right. \\
&&
\left.\left.\frac{\prod\limits_{k=1}^nr_k}{\left(\prod\limits_{k=1}^n(2\|\boldsymbol{w_k}\|_2/\|\boldsymbol{w}_k\|_1\sqrt{r_k\log(r_k)}+2\|\boldsymbol{w_k}\|_\infty/\|\boldsymbol{w}_k\|_1 r_k\log(r_k)+1)\right)-1}
\right\} \right)\nonumber.
\end{eqnarray}

Next we turn to the second term $(II)$. We write
\[(II) = 4\left(\prod\limits_{k=1}^n\|\boldsymbol{w_k}\|_1\right)\boldsymbol{\xi}^T\boldsymbol{\xi}=2\prod\limits_{k=1}^n(r_k\|\boldsymbol{w_k}\|_1)+4\left(\prod\limits_{k=1}^n\|\boldsymbol{w_k}\|_1\right)\left(\|\boldsymbol{\xi}\|_2^2-\frac{1}{2}\prod\limits_{k=1}^nr_k\right)\]
and bound the error term $4\left(\prod\limits_{k=1}^n\|\boldsymbol{w_k}\|_1\right)\left(\|\boldsymbol{\xi}\|_2^2-\frac{1}{2}\prod\limits_{k=1}^nr_k\right)$
with high probability. Observe that 
$\|\boldsymbol{\xi}\|_2^2-\frac{1}{2}\prod\limits_{k=1}^nr_k$ is a zero-mean subgaussian random variable, and thus  satisfies for all $t>0$  that
\[\mathbb{P}\left\{\left|\|\boldsymbol{\xi}\|_2^2-\frac{1}{2}\prod\limits_{k=1}^nr_k\right|\geq t\right\}\leq 2\exp\left(\frac{-c'' t^2}{\prod\limits_{k=1}^nr_k}\right)
\]
for some constant $c''$. Thus,  for any
$t>0$ we have
\[\mathbb{P}\left\{\left|4\left(\prod\limits_{k=1}^n\|\boldsymbol{w_k}\|_1\right)\left(\|\boldsymbol{\xi}\|_2^2-\frac{1}{2}\prod\limits_{k=1}^nr_k\right)\right|\geq t\right\}\leq 2\exp\left(\frac{-c^{''} t^2}{16\prod\limits_{k=1}^n(r_k\|\boldsymbol{w_k}\|_1^2)} \right).
\]

 Thus,
\begin{eqnarray}\label{eqn:(II)}
\mathbb{P}\left\{\left|(II)-2\prod\limits_{k=1}^n(r_k\|\boldsymbol{w_k}\|_1)\right|\geq \frac{1}{2}\prod\limits_{k=1}^nr_k\|\boldsymbol{w_k}\|_1\right\}&\leq&2\exp\left(\frac{-c^{''}}{64} \prod\limits_{k=1}^nr_k \right).
\end{eqnarray}

Combing \eqref{eqn:(I)} and \eqref{eqn:(II)}, we can conclude that with probability at  least \begin{multline*}
1 -2\exp\left(-c^{''}\prod\limits_{k=1}^nr_k\right)-2\exp\left(-c'\cdot\min\left\{ \frac{\prod\limits_{k=1}^nr_k}{\left(\prod\limits_{k=1}^n(2r_k(\|\boldsymbol{w}_k\|_2/\|\boldsymbol{w}_k\|_1)^2+1)\right)-1},\right.\right.\\
\left.\left.\prod\limits_{k=1}^n r_k,\frac{\prod\limits_{k=1}^nr_k}{\left(\prod\limits_{k=1}^n(2\|\boldsymbol{w_k}\|_2/\|\boldsymbol{w}_k\|_1\sqrt{r_k\log(r_k)}+2\|\boldsymbol{w_k}\|_\infty/\|\boldsymbol{w}_k\|_1 r_k\log(r_k)+1)\right)-1}
\right\} \right),
\end{multline*}
the following holds
\begin{eqnarray*}
\left\|\mathcal{H}\hadam\left(\mathcal{T}-\widetilde{\mathcal{T}}\right)\right\|_F^2&=&(I)+(II)\\
&\geq&2\prod\limits_{k=1}^n(r_k\|\boldsymbol{w_k}\|_1)-|II-2\prod\limits_{k=1}^n(r_k\|\boldsymbol{w_k}\|_1)|-(I)\\
&\geq& \prod\limits_{k=1}^n(r_k\|\boldsymbol{w_k}\|_1)=\left(\prod\limits_{k=1}^nr_k\right)\|\mathcal{H}\hadam\mathcal{B}\|_F^2.
\end{eqnarray*}

By a union of bound over all of the points in $S$, we establish items 1 and 3 of \mbox{the lemma}.
\end{proof}
Now we are ready to prove the lower bound in Theorem \ref{thm:bfsomega}. First we give a formal statement for the lower bound in Theorem \ref{thm:bfsomega} by introducing the constant $C'$ to  characterize the ``flatness'' of $\mathcal{W}$.
 \begin{theorem}[Lower bound for low-rank tensor when $\mathcal{W}$ is flat and $\Omega\sim \mathcal{W}$]\label{thm: lb}Let $\mathcal{W}=\boldsymbol{w_1}\out \cdots\out \boldsymbol{w_n}\in\mathbb{R}^{d_1\times \cdots\times d_n}$  be a CP rank-1 tensor so that all $(i_1,\cdots,i_n)\in[d_1]\times\cdots\times [d_n]$ with 
 $\|\mathcal{W}\|_\infty\leq 1$, so that
 \[\max_{i_k}|\boldsymbol{w_k}(i_k)| \leq C'\min_{i_k}|\boldsymbol{w_k}(i_k)|, \text{ for all }k=1,\cdots,n.
 \]
 
  Suppose that we choose each $(i_1,\cdots,i_n)\in[d_1]\times\cdots\times[d_n]$ independently with probability $\mathcal{W}_{i_1\cdots i_n}$ to form a set $\Omega\subseteq[d_1]\times\cdots\times[d_n]$. 
 Then with probability at least $1-\exp(-C\cdot m)$ over the choice of $\Omega$, the following holds:
 
Let $\sigma, \beta>0$ and let $K_{\boldsymbol{r}}\subseteq\mathbb{R}^{d_1\times \cdots\times d_n}$ be the cone of the tensor with Tucker rank \mbox{$\boldsymbol{r}=[r_1~\cdots ~r_n]$}. For any algorithm $\mathcal{A}:\mathbb{R}^{\Omega}\rightarrow\mathbb{R}^{d_1\times \cdots\times d_n}$ that takes as input $\mathcal{T}_\Omega+\mathcal{Z}_\Omega$ and outputs a guess $\widehat{\mathcal{T}}$ for $\mathcal{T}$, for $\mathcal{T}\in K_{\boldsymbol{r}}\cap\beta \mathbf{B}_\infty$ and $\mathcal{Z}_{i_1\cdots i_n}\sim\mathcal{N}(0,\sigma^2)$, then there is some $\mathcal{T}\in K_{\boldsymbol{r}}\cap\beta \mathbf{B}_{\infty}$ \mbox{so that}
\begin{eqnarray*}
&&\frac{\|\mathcal{W}^{(1/2)}\hadam(\mathcal{A}(\mathcal{T}_\Omega+\mathcal{Z}_\Omega)-\mathcal{T})\|_F}{\|\mathcal{W}^{(1/2)}\|_F}\\
&\geq& c\cdot\min\left\{\frac{\beta}{\sqrt{\log(8\prod\limits_{k=1}^nd_k)}},   \frac{\sigma }{\sqrt{|\Omega|}}\sqrt{\prod_{k=1}^nr_k}\cdot\min\left\{\sqrt{\frac{1}{\left(\prod\limits_{k=1}^n(1+2C'^2r_k/d_{k})\right)-1}},\right.\right.\\
&&\left.\left.1,\sqrt{\frac{1}{\left(\prod\limits_{k=1}^n(2C'\sqrt{r_k/d_k\log(r_k)}+2C' r_k/d_k\log(r_k)+1)\right)-1}} \right\} \right\},
\end{eqnarray*}
with probability at least $\frac{1}{2}$ over the randomness of $\mathcal{A}$ and the choice of $\mathcal{Z}$. Above $c$, $C$  are constants which depend only on $C'$.
\end{theorem}

\begin{proof}
 Let $m=\|\mathcal{W}^{(1/2)}\|_F^2=\prod\limits_{k=1}^n\|\boldsymbol{w_k}\|_1$, so that $\mathbb{E}|\Omega|=m$.

We instantiate Lemma \ref{lmm:lb} with $\mathcal{H}=\mathcal{W}^{(1/2)}$ and $\mathcal{B}$ being the tensor whose entries are all $1$. Let $S$ be the set guaranteed by Lemma \ref{lmm:lb}. We have \[\max_{\mathcal{T}\in S}\|\mathcal{T}\|_\infty\leq \sqrt{\frac{1}{2}\log\left(8\prod\limits_{k=1}^nd_k\right)\prod\limits_{k=1}^nr_k}.
\] and
\[\max_{\mathcal{T}\in S}\|\mathcal{T}_\Omega\|_F\leq 2\sqrt{\prod\limits_{k=1}^nr_k}\|\mathcal{B}_\Omega\|_F=2\sqrt{|\Omega|\prod\limits_{k=1}^nr_k}.
\]

We also have
\[\|\mathcal{W}^{(1/2)}\hadam(\mathcal{T}-\mathcal{T}')\|_F\geq\sqrt{\prod\limits_{k=1}^nr_k}\|\mathcal{W}^{(1/2)}\|_F=\sqrt{m\prod\limits_{k=1}^nr_k}
\]
for $\mathcal{T}\neq\mathcal{T}'\in S$. 
{Using the assumption that $\boldsymbol{w}_k$ are flat, the size of the set $S$ is bigger than or equal to  
\begin{eqnarray*}
N&=&C\exp\left({c\cdot\min\left\{ \frac{\prod\limits_{k=1}^nr_k}{\left(\prod\limits_{k=1}^n(2r_k(\|\boldsymbol{w}_k\|_2/\|\boldsymbol{w}_k\|_1)^2+1)\right)-1},\prod\limits_{k=1}^n r_k,\right.}\right.\\
 &&\left. \left.\frac{\prod\limits_{k=1}^nr_k}{\left(\prod\limits_{k=1}^n(2\|\boldsymbol{w_k}\|_2/\|\boldsymbol{w}_k\|_1\sqrt{r_k\log(r_k)}+2\|\boldsymbol{w_k}\|_\infty/\|\boldsymbol{w}_k\|_1 r_k\log(r_k)+1)\right)-1}\right\} \right)\\
&\geq&C\exp\left(c\cdot\min\left\{ \frac{\prod\limits_{k=1}^nr_k}{\left(\prod\limits_{k=1}^n(2C'^2r_k/d_k+1)\right)-1},\prod\limits_{k=1}^nr_k,\right.\right.\\
&&\left.\left.\frac{\prod\limits_{k=1}^nr_k}{\left(\prod\limits_{k=1}^n(2C'\sqrt{r_k\log(r_k)/d_k}+2C' r_k\log(r_k)/d_k+1)\right)-1}
\right\} \right)
\\
 &\geq&\exp\left(C''\cdot\min\left\{ \frac{\prod\limits_{k=1}^nr_k}{\left(\prod\limits_{k=1}^n(2C'^2r_k/d_k+1)\right)-1},\prod\limits_{k=1}^nr_k,\right.\right.\\
 &&\left.\left.\frac{\prod\limits_{k=1}^nr_k}{\left(\prod\limits_{k=1}^n(2C'\sqrt{r_k\log(r_k)/d_k}+2C' r_k\log(r_k)/d_k+1)\right)-1}
\right\}  \right),
\end{eqnarray*}
where $C''$ depends on $c$ and $C$. }
Set \begin{multline*}
\kappa=\min\left\{
\frac{\beta}{\sqrt{\frac{1}{2}\log(8\prod\limits_{k=1}^nd_k)\prod\limits_{k=1}^nr_k}}, \frac{\sigma \sqrt{C''}}{8\sqrt{|\Omega|}}\sqrt{\frac{\prod\limits_{k=1}^nd_k}{(\prod\limits_{k=1}^n(d_k+2C'^2r_k))-\prod\limits_{k=1}^nd_k}}, \frac{\sigma \sqrt{C''}}{8\sqrt{|\Omega|}},\right.\\
\left.\frac{\sigma \sqrt{C''}}{8\sqrt{|\Omega|}}\sqrt{\frac{\prod\limits_{k=1}^nd_k}{(\prod\limits_{k=1}^n(2C'\sqrt{d_kr_k\log(r_k)}+2C' r_k\log(r_k)+d_k))-\prod\limits_{k=1}^nd_k}}~\right\}.
\end{multline*}

Observe that $\frac{\sigma\sqrt{\log|S|}}{4\max_{\mathcal{T}\in S}\|\mathcal{T}_\Omega\|_F}
\geq\frac{\sigma\sqrt{\log(N)}}{4\max_{\mathcal{T}\in S}\|\mathcal{T}_\Omega\|_F}$ and
\begin{eqnarray*}
&&\frac{\sigma\sqrt{\log(N)}}{4\max_{\mathcal{T}\in S}\|\mathcal{T}_\Omega\|_F}\\
&\geq&\frac{\sigma \sqrt{C''}}{8\sqrt{|\Omega|\prod\limits_{k=1}^nr_k}}\cdot\min\left\{ \frac{\prod\limits_{k=1}^nr_k}{\left(\prod\limits_{k=1}^n(2C'^2r_k/d_k+1)\right)-1},\prod\limits_{k=1}^nr_k,\right.\\
&&\left.\frac{\prod\limits_{k=1}^nr_k}{\left(\prod\limits_{k=1}^n(2C'\sqrt{r_k\log(r_k)/d_k}+2C' r_k\log(r_k)/d_k+1)\right)-1}
\right\}\\
&=&\frac{\sigma \sqrt{C''}}{8\sqrt{|\Omega|}}\cdot\min\left\{ \sqrt{\frac{\prod\limits_{k=1}^nd_k}{(\prod\limits_{k=1}^n(d_k+2C'^2r_k))-\prod\limits_{k=1}^nd_k}}, 1,\right.\\
&&\left.\sqrt{\frac{\prod\limits_{k=1}^nd_k}{(\prod\limits_{k=1}^n(2C'\sqrt{d_kr_k\log(r_k)}+2C' r_k\log(r_k)+d_k))-\prod\limits_{k=1}^nd_k}} 
~\right\}\geq\kappa,
\end{eqnarray*}
so this is a legitimate choice of $\kappa$ in Lemma \ref{lmm:FanoE}. Next, we verify that $\kappa S\subseteq K\cap \beta \mathbf{B}_\infty$. Indeed, we have
\[\kappa\max_{\mathcal{S}}\|\mathcal{T}\|_\infty\leq\kappa \sqrt{\frac{1}{2}\log(8\prod\limits_{k=1}^nd_k)\prod\limits_{k=1}^nr_k}\leq\beta,
\]
so $\kappa S\subseteq\beta \mathbf{B}_\infty$, and every element of $\mathcal{S}$ has Tucker rank $\boldsymbol{r}$ by construction.

Then Lemma \ref{lmm:FanoE} concludes that if $\mathcal{A}$ works on $K_{\boldsymbol{r}}\cap\beta \mathbf{B}_\infty$, then there is a tensor $\mathcal{T}\in K_{\boldsymbol{r}}\cap \beta \mathbf{B}_\infty$ so that
\begin{eqnarray*}
&&\|\mathcal{W}^{(1/2)}\hadam(\mathcal{A}(\mathcal{T}_\Omega+\mathcal{Z}_\Omega)-\mathcal{T})\|_F\\
&\geq&\frac{\kappa}{2}\min_{\mathcal{T}\neq\mathcal{T}'\in S}\|\mathcal{W}^{(1/2)}\hadam(\mathcal{T}-\mathcal{T}')\|_F\\
&\geq&\frac{1}{2}\min\left\{\frac{\beta}{\sqrt{\frac{1}{2}\log(8\prod\limits_{k=1}^nd_k)\prod\limits_{k=1}^nr_k}},\frac{\sigma \sqrt{C''}}{8\sqrt{|\Omega|}}\sqrt{\frac{\prod\limits_{k=1}^nd_k}{(\prod\limits_{k=1}^n(d_k+2C'^2r_k))-\prod\limits_{k=1}^nd_k}}, \frac{\sigma \sqrt{C''}}{8\sqrt{|\Omega|}},\right.\\
&&\left.\frac{\sigma \sqrt{C''}}{8\sqrt{|\Omega|}}\sqrt{\frac{\prod\limits_{k=1}^nd_k}{(\prod\limits_{k=1}^n(2C'\sqrt{d_kr_k\log(r_k)}+2C' r_k\log(r_k)+d_k))-\prod\limits_{k=1}^nd_k}}\right\}\sqrt{m\prod\limits_{k=1}^nr_k}\\
&=&\min\left\{  \frac{\beta\sqrt{m}}{\sqrt{2\log(8\prod\limits_{k=1}^n d_k)}}, \frac{\sigma \sqrt{C''m}}{16\sqrt{|\Omega|}}\sqrt{\prod\limits_{k=1}^nr_k}\cdot\min\left\{\frac{1}{\sqrt{\left(\prod\limits_{k=1}^n(1+2C'^2r_k/d_k)\right)-1}},\right.\right.\\
&&\left.\left.1,\frac{1}{\sqrt{\left(\prod\limits_{k=1}^n(2C'\sqrt{r_k/d_{k}\log(r_k)}+2C' r_k/d_{k}\log(r_k)+1)\right)-1}}\right\}\right\}.
\end{eqnarray*}

Additionally, by Lemma \ref{lmm:est_omega}, we conclude that
\begin{eqnarray*}
&&\frac{\|\mathcal{W}^{(1/2)}\hadam(\mathcal{A}(\mathcal{T}_\Omega+\mathcal{Z}_\Omega)-\mathcal{T})\|_F}{\|\mathcal{W}^{(1/2)}\|_F}\\
&\geq&\tilde{c}\cdot\min\left\{\frac{\beta}{\sqrt{\log(8\prod\limits_{k=1}^nd_k)}},   \frac{\sigma }{\sqrt{|\Omega|}}\sqrt{\prod_{k=1}^nr_k}\cdot\min\left\{\frac{1}{\sqrt{\left(\prod\limits_{k=1}^n(1+2C'^2r_k/d_{k})\right)-1}},\right.\right.\\
&&\left.\left.1,\frac{1}{\sqrt{\left(\prod\limits_{k=1}^n(2C'\sqrt{r_k/d_k\log(r_k)}+2C' r_k/d_k\log(r_k)+1)\right)-1}}\right\}\right\},
\end{eqnarray*}
where $\tilde{c}$ depends on the above constants.
\end{proof}
\begin{Remark}
Consider the special case when $\mathcal{T}\in\mathbb{R}^{d_1\times d_2}$ with $d_1\leq d_2$. Then we can consider the reconstruction of $S$ in Lemma \ref{lmm:lb} with $\mathcal{H}=\mathcal{W}^{(1/2)}$, $\mathcal{B}$ being the tensor whose entries are all 1, $\mathcal{C}\in\{\pm 1\}^{r\times d_2}$, ${}^1U\in\{\pm 1\}^{d_1\times r}$ and ${}^2U\in\{\pm 1\}^{d_2\times d_2}$ which implies that $r_1=r$ and $r_2=d_2$. Thus, we have
\begin{eqnarray*}
&&\frac{\|\mathcal{W}^{(1/2)}\hadam(\mathcal{A}(\mathcal{T}_\Omega+\mathcal{Z}_\Omega)-\mathcal{T})\|_F}{\|\mathcal{W}^{(1/2)}\|_F}
\geq\tilde{c}\cdot\min\left\{\frac{\sigma }{\sqrt{|\Omega|}}\sqrt{rd_2},\frac{\beta}{\sqrt{\log(8d_1d_2)}}\right\},
\end{eqnarray*}
which has the same bound as the one in (\cite{foucart2019weighted} Lemma 28).
\end{Remark}

\section*{Acknowledgements}The authors are supported by NSF CAREER DMS 1348721 and NSF BIGDATA 1740325.
 The  authors take pleasure in thanking Hanqin Cai,  Keaton Hamm, Armenak Petrosyan, Bin Sun, and  Tao Wang for comments and suggestions on  the manuscript.  

\bibliographystyle{plain}
\bibliography{tensor}

\end{document}